\documentclass[a4paper,12pt]{article}
\usepackage{amsfonts,amsthm,amssymb,amsmath}
\usepackage[colorlinks=true,linkcolor=blue,citecolor=blue]{hyperref}

\usepackage{tcolorbox}
\usepackage[english]{babel}
\usepackage{amssymb,amsmath}
\usepackage{xcolor}

\newtheorem{Theo}{Theorem}[section]
\newtheorem{Prop}[Theo]{Proposition}
\newtheorem{Coro}[Theo]{Corollary}
\newtheorem{Lemm}[Theo]{Lemma}

\newtheorem{Conj}[Theo]{Conjeture}

\newtheorem{Rema}[Theo]{Remark}

\newcommand{\xo}{\vec{X}}
\newcommand{\yo}{\vec{Y}}

\DeclareMathOperator{\mon}{mon}

\DeclareMathOperator{\Maj}{Maj}

\DeclareMathOperator{\supp}{supp}

\def\T{\mathbb{ T}}
\def\N{\mathbb{ N}}
\def\R{\mathbb{ R}}

\begin{document}
\title{Variants of a multiplier theorem of Kislyakov}
\author{A.~Defant, M.~Masty{\l}o and A.~P\'{e}rez~Hern\'{a}ndez}

\date{}

\maketitle
\noindent
\renewcommand{\thefootnote}{\fnsymbol{footnote}}
\footnotetext{The second named author was supported by the National Science  Centre, Poland,
Grant no.~2019/33/B/ST1/00165. The third named author acknowledges support from the grants “Juan de
la Cierva Formación” (FJC2018-036519-I) and 2021-MAT11 (ETSI Industriales,
UNED).}

\begin{abstract}
\noindent
We prove stronger variants of a multiplier theorem of Kislyakov.~The key \linebreak ingredients are based on ideas of Kislaykov
and the Kahane--Salem--Zygmund inequality. As a by-product we show various multiplier theorems for spaces of trigonometric
polynomials on the $n$-dimensional torus $\mathbb{T}^n$ or Boolean cubes $\{-1,1\}^N$. Our more abstract approach based on
local Banach space theory has the advantage that it allows to consider more general compact abelian groups instead of only
the multidimensional torus.~As an application we show that  various  recent  $\ell_1$-multiplier theorems for trigonometric
polynomials in several variables or ordinary Dirichlet series may be proved without the Kahane--Salem--Zygmund inequality.
\end{abstract}

\newpage

\tableofcontents

\section{Introduction}

Let $\T$ be  the torus in the complex plane, that is, the compact abelian group of all $z \in \mathbb{C}$ with $|z|=1$, which
carries  the normalized Lebesgue measure $\nu$ on $\T$ as its Haar measure. By $\T^\infty$ we denote  the countable product
of $\T$, which  again forms a~compact abelian group (the Haar measure is the countable product of $\nu$), and identify its dual
group with $\mathbb{Z}^{(\mathbb{N})}$, all finite multi indices $\alpha \in \mathbb{Z}^n,\,n \in\mathbb{N}$. We write
$\mathbb{N}_0^{(\mathbb{N})}$ for all $\alpha \in \mathbb{Z}^{(\mathbb{N})}$ with entries in $\N_0$.

As usual $H_\infty(\T^\infty)$ stands for  the Banach space of  all functions $f \in L_\infty(\T^\infty)$ such that
$\widehat{f}(\alpha)= 0$ for all $\alpha =(\alpha_i) \in \mathbb{Z}^{(\mathbb{N})}$ with $\alpha_j<0$ for some $j$. We call
$H_\infty(\T^\infty)$ Hardy space on the infinite dimensional torus, and  denote its closed subspace of all continuous
functions by $C^A(\T^\infty)$.

A sequence $\xi = (\xi_\alpha)_{\alpha \in \mathbb{N}_0^{(\mathbb{N})}} $  of scalars is a bounded $\ell_1(\mathbb{N}_0^{(\mathbb{N})})$-multiplier
of $C^A(\mathbb{T}^\infty)$, whenever the mapping $ M_\xi\colon C^{A}(\mathbb{T}^\infty) \to \ell_1(\mathbb{N}_0^{(\mathbb{N})})$
given by
\[
M_\xi(f):= \big(\widehat{f}(\alpha) \xi_{\alpha}\big)_{\alpha \in \mathbb{Z}^{(\mathbb{N})}}, \quad\, f\in C^{A}(\mathbb{T}^\infty)
\]
is bounded. The  following necessary condition for  such multipliers  is due  to Kislyakov \cite[Theorem~6]{Kisl}, and it in fact
is the main motivation of this paper.

\begin{Theo} \label{Thm:KKKK}
Let $\xi = (\xi_\alpha)_{\alpha \in \mathbb{N}_0^{(\mathbb{N})}} $ be a bounded $\ell_1(\mathbb{N}_0^{(\mathbb{N})})$-multiplier of $C^A(\mathbb{T}^\infty)$.
Then
\begin{align} \label{start}
\sup_{n,d \in \mathbb{ N}}\,\,\,\frac{1}{\sqrt{n \log( 1+dn)}}
\Big(\sum_{\alpha \in \mathbb{N}_0^n \colon \max\{\alpha_1,...,\alpha_n\}\leq d} & |\xi_{\alpha}|^2\Big)^{1/2}
\,\,\,< \infty\,.
\end{align}
\end{Theo}

\medskip

Inspired by this result, in particular  the techniques from local Banach space theory which Kislyakov uses to prove it,
we study the following more general (but also more vague) question:

Let $G$ is a compact abelian group with Haar measure $\nu$, and $\Gamma$ a subset of characters in the dual group
$\widehat{G}$. Moreover, let $X(\Gamma)$ be a Banach sequence space over the set $\Gamma$. A (real or complex) sequence
$\xi = (\xi_\gamma)_{\gamma \in \Gamma}$ is an  $X(\Gamma)$-multiplier of a~closed subspace  $V \subset L_p(G)$,
whenever
\[
(\widehat{f}(\gamma) \xi_\gamma) \in X(\Gamma), \quad\, f \in V\,.
\]
The problem then is to find necessary and sufficient conditions for such $X(\Gamma)$-multipliers of $V$.
Our applications mainly focus on multiplier theorems for the $n$-dimensional torus  $\mathbb{T}^n$, its Boolean
counterpart, the  $n$-dimensional Boolean cube $\{-1,1\}^n$, as well as their countable counterparts
$\mathbb{T}^\infty := \mathbb{T}^\mathbb{N}$ and $\{-1,1\}^\infty:= \{-1, 1\}^{\mathbb{N}}$.

\medskip

Observe that by a simple closed graph argument the study of $X(\Gamma)$-multipliers for $V \subset L_\infty(G)$
means to study concrete  inequalities: $\xi$ is an  $X(\Gamma)$-multiplier for $V \subset L_\infty(G)$ if and only
if there is a constant $C=C(\xi)~>~0$ such that
\[
\sum_{\gamma \in \Gamma} |\widehat{f}(\gamma) \xi_\gamma|\leq C \|f\|_\infty, \quad\, f\in V\,.
\]
We note that only in very few cases a full description of the set of all $X$-multipliers~$\xi$ of $V \subset L_\infty(G)$
is possible. In most of our applications,  we are  able to give necessary and/or sufficient conditions in term of the
asymptotic decay of~$\xi$.

In the first section we show that  Theorem~\ref{Thm:KKKK} is in fact a consequence of the Kahane-Salem-Zygmund inequality,
and in  Theorem~\ref{Thm:SK} and Corollary~\ref{Thm:K} we  extend Kislyakov's multiplier theorem  to certain analytic
subspaces of $L_\infty(\mathbb{T}^\infty)$ instead of $C^A(\T^\infty)$. It should be mentioned here that Kislyakov's
approach to Theorem~\ref{Thm:KKKK} is different, and  in the  second and third section we analyze his cycle of ideas from
local Banach space theory -- the main advantage is that they apply to more general compact abelian groups than only
multi-dimensional tori.

The crucial  link, which makes this possible,  comes from  Lemma~\ref{oneone} showing that, given a~compact abelian group
$G$, a~finite subset $\Gamma$  in $\widehat{G}$, and a~Banach space $F:=(\mathbb{C}^{\Gamma}, \|\cdot\|)$, for every finite
sequence $\xi = (\xi_\gamma)_{\gamma \in \Gamma}$  one has
\[
\pi_2(M_\xi: \mathcal P_\Gamma \to F)
=\sup_{\|\mu\|_{\ell_2(\Gamma)}\leq 1} \|(\mu_{\gamma} \xi_\gamma)_{\gamma \in \Gamma}\|_F\,,
\]
where $\mathcal P_\Gamma$ stands for the Banach space of all finite polynomials $\sum_{\gamma \in \Gamma}\xi_\gamma \gamma$
endowed with the sup norm, $M_\xi$  for the  multiplier, which assigns to every finite polynomial
$\sum_{\gamma \in \Gamma}\xi_\gamma \gamma$ the finite sequence $(\mu_{\gamma} \xi_\gamma)_{\gamma \in \Gamma} \in F$,
and $\pi_2(M_\xi)$ for the $2$-summing norm of this operator.

In the Theorems~\ref{Thm:KK}, ~\ref{Thm:IntMultB}, and~\ref{Thm:IntMult}, fundamental knowledge on $2$-summing operators
leads  to improvements of Theorem~\ref{Thm:KKKK}.

In the last section we apply our results to study multiplier theorems for spaces of functions on multidimensional
tori and Boolean cubes. We focus on topics like Sidon constants, Bohr radii, monomial convergence, as well as
Dirichlet series.

Using Kislyakov's ideas we prove new results, but we also reprove recent known results, which were originally
proved through the use of the Kahane--Salem--Zygmund inequality.

At first glance this might look  surprising, but on the other hand we already remarked that  our starting point,
Theorem~\ref{Thm:KKKK}, is a~consequence of the Kahane-Salem-Zygmund inequality, and conversely we showed in the
recent article~\cite{DefantMastylo} that Kislyakov's ideas are of great relevance within a~further study of the
Kahane-Salem-Zygmund inequality.

\section{Preliminaries}

\noindent{\bf Banach spaces.}
Let $X$, $Y$ be Banach spaces. We denote by $B_X$ the closed unit ball of $X$, and by $X^{*}$ its dual Banach space.
If we write $X\hookrightarrow Y$, then we assume that $X\subset Y$ and the inclusion map ${\rm{id}}\colon X \to Y$
is bounded. If $X=Y$ with equality of norms, then we write $X\cong Y$. As usual $C(K)$ denotes the Banach space of
all continuous functions on a~compact Hausdorff space $K$, with the sup norm $\|\cdot\|_\infty$.

We denote by $L(X, Y)$ the space of all bounded linear operators $T \colon X\to Y$ with the usual operator norm.
An operator $T \in L(X, Y)$ is said to be an isomorphic embedding of $X$ into $Y$ whenever there exists $C>0$ such
that $\|Tx\|_Y \geq C \|x\|_X$ for every $x\in X$. Thus $T^{-1}$ is an isomorphism from $(TX, \|\cdot\|_Y)$ onto
$X$. Given a~real number $1\leq \lambda <\infty$, we say that $X$ $\lambda$-embeds into $Y$ whenever there exists
an isomorphic embedding $T$ of $X$ into $Y$ such that $\|T\|\,\|T^{-1}\| \leq \lambda$\,. In this case, we call $T$
a~$\lambda$-embedding of $X$ into $Y$.

Let $\Gamma$ be a nonempty set. We denote by $\ell_\infty(\Gamma)$ the space of all bounded functions on $\Gamma$,
endowed with the sup norm. For any $\xi, \eta \in \ell_{\infty}(\Gamma)$, we denote by $\xi \cdot \eta$ their
pointwise product. Let $E$ and $F$ be linear  subspaces of $\mathbb{K}^{\Gamma}$, where $\mathbb{K} = \mathbb{R}$
or $\mathbb{K} = \mathbb{C}$. Any $\xi \in \mathbb{K}^\Gamma$, such that $\xi \eta:=\xi \cdot \eta \in F$ for all
$\eta \in E$, defines a   \emph{diagonal operator}  $D_\xi\colon E \to F$  given by $D_\xi(\eta):= \xi \eta$ for
all $\eta \in E$. We write $\mathcal{D}(E,F)$ for the vector space of all such maps. If $E$ and $F$ are Banach
spaces such that the inclusion maps from $E$ and $F$ into $\mathbb{K}^{\Gamma}$ are continuous, then
$\mathcal{D}(E, F)$ equipped with the norm
\[
\|D_{\xi}\|_{\mathcal{D}(E, F)} := \sup_{\|(\eta_\gamma)\|_E \leq 1} \|(\xi_\gamma\eta_\gamma)\|_F
\]
is a~Banach space.

We apply standard methods from  local Banach space theory. Recall that a~Banach space $X$ has cotype $2$
whenever there is a~constant $C$ such that for each choice of finitely many $x_1, \ldots,x_n \in X$
\[
\Big(\sum_{k=1}^n \|x_k\|^2 \Big)^{\frac{1}{2}}
\leq C \Big(\int_0^1\Big\|\sum_{k=1}^n  r_k(t) x_k \Big\|^2 dt\Big)^{\frac{1}{2}}\,,
\]
where $r_k$ denotes the $k$th Rademacher function. The least possible value of this constant is denoted $C_2(X)$.

An operator $T\colon X\to Y$ between Banach spaces is said to be absolutely $p$-summing ($p$-summing for short)
with $1\leq p<\infty$ if there is a~constant $C>0$ such that, for each $n\in \mathbb{N}$ and for all sequences
$(x_k)_{k=1}^n$ in $X$, we have
\[
\Big(\sum_{k=1}^n \|Tx_k\|_{Y}^p\Big)^{1/p}
\leq C \sup_{x^{*}\in B_{X^{*}}} \Big(\sum_{k=1}^n |x^{*}(x_k)|^p\Big)^{\frac{1}{p}}\,.
\]
The least such constant $C$ is denoted by $\pi_p(T\colon X \to Y)$ ($\pi_p(T)$ for short), and is called the
absolutely $p$-summing norm of $T$. We refer to the theory of $p$-summing operators to \cite{DiestelJarchowTonge}
and \cite{Pisier}.

We recall that if $K$ is a~compact Hausdorff space, $X$ a~closed subspace of $C(K)$, and $Y$ a~Banach space, then
the so-called Pietsch domination theorem states that for every  $p$-summing $T\colon X\to Y$,  there exists a~probability
Borel measure $\mu$ on $K$ such that
\[
\|Tf\|_Y  \leq \pi_p(T) \Big(\int_K |f|^p\,d\mu \Big)^{\frac{1}{p}}, \quad\, f\in X\,.
\]

\medskip

\noindent{\bf Compact abelian groups.}
In the following we fix some compact abelian group $G := (G, \cdot)$. A~linear subspace $X$ of $\mathbb{K}^{G}$ is said
to be translation invariant whenever  for every $f\in X$ and every $h\in G$, the translation $f_h \in X$, where
$f_h(g):= f(g\cdot h)$ for every $g\in G$.

As usual we write  $\widehat{G}$ for the dual group of $G$ (i.e., the set of all continuous characters on $G$), and we
denote by $\nu$ the (normalized) Haar measure on $G$, a~unique translation invariant regular Borel probability measure.
Recall that the translation invariance of $\nu$ is equivalent to the formula
\[
\int_G f(g)\,d\nu(g) = \int_G f_h(g)\,d\nu(g), \quad\, f\in L_1(G, \nu), \,\, h\in G\,.
\]
The following well-known result from \cite{Pel} is central for our purposes; for the sake of completeness we include
a~simple proof.

\begin{Lemm}
\label{one}
Let $G$ be a~compact abelian group with normalized Haar measure $\nu$, let $X$ be a~closed translation invariant
subspace of $C(G)$, and let $Y$ be an arbitrary Banach space. Suppose that  $T\colon X \to Y$ is a~$p$-summing
operator which satisfies that $\|Tf_h\|_Y = \|Tf\|_Y$ for all $f\in X$, $h\in G$. Then
\[
\|Tf\|_Y \leq \pi_p(T)\Big(\int_G |f(g)|^p\,d\nu(g) \Big)^{1/p}, \quad\, f\in X\,.
\]
\end{Lemm}

\noindent

\begin{proof}
Pietsch's domination theorem (mentioned above) combined with Fubini's theorem shows that there is a~probability Borel
measure $\mu$ on $G$ such that for all $f\in X$
\begin{align*}
\|Tf\|_{Y}^{p} & = \int_G \|Tf_h\|_{Y}^p \,d\nu(h) \leq \pi_p(T)^p \int_G\Big(\int_G |f_h(g)|^p\,d\mu(g)\Big)\,d\nu(h) \\
& = \pi_p(T)^p  \int_G\Big(\int_G |f_h(g)|^p\,d\nu(h) \Big)\,d\mu(g)\,.
\end{align*}
Since the Haar measure $\nu$ is translation invariant, it follows that
\begin{align*}
\|Tf\|_{Y}^{p} & = \pi_p(T)^p  \int_G\Big(\int_G |f_g(h)|^p\,d\nu(h) \Big)\,d\mu(g) \\
& = \pi_p(T)^p  \int_G\Big(\int_G |f(h)|^p\,d\nu(h) \Big)\,d\mu(g) = \pi_p(T)^p \int_G |f(h)|^p\,d\nu(h)\,. \,\, \qedhere
\end{align*}
\end{proof}

\smallskip

\noindent As usual,  the Fourier transform of $f\in L_1(G, \nu)$ is defined  by
\[
\widehat{f}(\gamma) := \int_G f(g) \overline{\gamma(g)}\,d\nu, \quad\, \gamma \in \widehat{G}\,.
\]
Let $X$ be any subspace of $L_1(G, \nu)$, and  $\Gamma \subset \widehat{G}$  a~nonempty subset. Then
\[
X_\Gamma : = \{f\in X \colon\, \widehat{f}(\gamma)= 0 \quad\, \text{for all\, $\gamma \notin \Gamma$}\}\,.
\]
Clearly, $X_\Gamma$ is a translation invariant subspace of $X$. Note that for $X=C(G)$ or $X=L_p(G, \nu)$
with $1\leq p<\infty$, every translation invariant subspace of $X$ has the form $X_\Gamma$ for some
$\Gamma\subset \widehat{G}$.

In what follows,  for simplicity of notation, we write $C_\Gamma$ instead of $C(G)_\Gamma$.
By  $\widehat{C_\Gamma} \subset \mathbb{K}^\Gamma$ we denote the linear space of all $(\widehat{f}(\gamma))_{\gamma \in \Gamma},
\, f \in C_\Gamma$, which equipped with the norm
\[
\|(\widehat{f}(\gamma))_{\gamma \in \Gamma}\|_{\widehat{C_\Gamma}}:= \|f\|_{C(G)}, \quad\, f\in C_\Gamma
\]
forms a Banach space. Throughout the paper, if $\xi \in \mathbb{K}^\Gamma$ and $F=(\mathbb{K}^\Gamma,\|\cdot\|)$ are
Banach spaces, then the mapping $M_\xi\colon C_\Gamma \to F$ (which we call  multiplier) is given by
\[
M_\xi f:= (\xi_\gamma \widehat{f}(\gamma))_{\gamma \in \Gamma}, \quad\, f\in C_\Gamma\,.
\]
The space of of all such  multipliers is denoted by $\mathcal{M}(C_\Gamma, F)$, and it obviously identifies with
the space $\mathcal{D}(\widehat{C_\Gamma}, F)$ of all diagonal operators from $\widehat{C_\Gamma}$ into $F$.

A~subset $\Gamma$ of $\widehat{G}$ is called a~$p$-Sidon set ($1\leq p<\infty$) if there is a~constant $C$ such that
\[
\Big(\sum_{\gamma \in \Gamma} |\widehat{f}(\gamma)|^p\Big)^{\frac{1}{p}} \leq C \|f\|_{\infty}, \quad\, f\in C_\Gamma\,.
\]
The least possible value of this constant is denoted $S_p(\Gamma)$ and  called the $p$-Sidon constant of $G$.

Let $\{-1,1\}$ be the compact discrete group with the Haar measure $\sigma_{1}(\{-1\}) = \sigma_{1}(\{1\}) = 1/2$, and
let $\mathbb{T}$ be unit circle equipped with normalized Lebesgue measure. We will primarily be interested in the case
when $G=\mathbb{T}^n$, $G= \{-1, 1\}^{N}$,
 $G=\{-1,1\}^{\infty}$ or $G = \mathbb{T}^{\infty}$.

The countable product $\{-1,1\}^{\infty}$ is a~compact abelian group with the pointwise  product and the product topology. The
Haar measure on $\{ -1,1\}^{\infty}$ is then  just the countable product of the measure $\sigma_{1}$. Clearly,
the dual group of $\{-1,1\}^{\mathbb{N}}$ is the set $\{\chi_{S}\colon S \in \mathcal{P}_{fin}(\mathbb{N})\}$, where
$S\in \mathcal{P}_{fin}(\mathbb{N})$ means that $|S|:={\rm{card}}(S)<\infty$, and
$\chi_{S}\colon \{ -1,1\}^{\infty} \rightarrow \{-1,1\}$ is defined by $\chi_{S}(x) = \prod_{n \in S}{x_{n}}$ for
all $x = (x_{n})_{n \in \mathbb{N}}$ in $\{ -1,1\}^{\infty}$.

The compact abelian group $\mathbb{T}^{\infty}$ carries  the pointwise product, the product topology, and the product
of the normalized Lebesgue measure on $\mathbb{T}^{\infty}$ as its Haar measure. Denote by $\mathbb{Z}^{(\mathbb{N})}$
the set of all sequences $\alpha=(\alpha_1, \ldots, \alpha_n, \ldots)$ of integers which vanish for $n$ large enough.
Then we have $\widehat{\mathbb{T}^\infty}=\mathbb{Z}^{(\mathbb{N})}$, where each $\alpha \in \mathbb{Z}^{(\mathbb{N})}$
is identified with the character $\gamma(z) = z^{\alpha}:=\prod_{j=1}^{\infty} z^{\alpha_j}$, $z \in \mathbb{T}^\infty$.

The subset of all sequences  $\alpha \in \mathbb{Z}^{(\mathbb{N})}$, for which all entries are either 0 or natural,
is denote by $\mathbb{N}_0^{(\mathbb{N})}$. We write $\Lambda^\leq(m,n) \subset \mathbb{N}_0^{(\mathbb{N})}$ for
the subset of all  $\alpha$ of length $n$ and with order $|\alpha| = \sum_{j=1}^{n}\alpha_j \leq m$, whereas $\Lambda^=(m,n)$
consists of all $\alpha$'s of length $n$ but  with order $|\alpha| = m$.

\medskip

\noindent{\bf Trigonometric polynomials.} \label{P1}
Given $n \in \mathbb{N}$ and $m \in \mathbb{N}_0$, we write $T(m,n)$ for the set of all multi indices in
$\{\alpha \in \mathbb{Z}^n\colon |\alpha| \leq m\}$, and  $\mathcal{T}_{\leq m}(\mathbb{T}^n)$ for the space of all
trigonometric polynomials
\[
P(z) = \sum_{\alpha \in T(m,n)} c_\alpha z^{\alpha}, \quad\, z\in \mathbb{T}^n
\]
on the $n$-dimensional torus $\mathbb{T}^n$ which have degree $\text{deg}(P) = \max \{|\alpha|; \, c_\alpha \neq 0\} \leq m$.
Clearly, $\mathcal{T}_{\leq m}(\mathbb{T}^n)$ together with the sup norm $\|\cdot\|_{\mathbb{T}^n}$ (also denoted by
$\|\cdot\|_{\infty}$) forms a Banach space

By $\mathcal{P}_{\leq m}(\mathbb{T}^n)$ we denote the closed subspace of $\mathcal{T}_{\leq m}(\mathbb{T}^n)$ of all
trigonometric analytic polynomials $P(z)= \sum_{\alpha \in \Lambda^{\leq}(m,n)} c_\alpha z^\alpha$ for all $z\in \mathbb{T}^n$.
The space $\mathcal{P}_{= m}(\mathbb{T}^n)$ is defined to be the closed subspace of all $m$-homogeneous polynomials $P$ given
by $P(z)= \sum_{\alpha \in \Lambda^{=}(m,n)} c_\alpha z^\alpha$.

Moreover, we are going to make use of the so-called 'hypercontractive' Bohnen\-blust-Hille inequality: There is some universal
constant $C >0$ such that, for each $m,n$ and $P \in \mathcal{P}_{\leq m}(\mathbb{T}^n)$
\begin{align} \label{BH}
\Big(\sum_{\alpha \in \Lambda^{\leq}(m,n)} |\widehat{P}(\alpha)|^{\frac{2m}{m+1}}\Big)^{\frac{m+1}{2m}}
\leq C^{\sqrt{m \log m}} \|P\|_\infty\,;
\end{align}
we refer to  \cite{BoHi31} (original form), \cite{DefantFrerickOrtegaOunaiesSeip} (hypercontractive form), \cite{Bohrradius}
(subexponential form), and within the context of Dirichlet series and holomorphic functions in infinitely many variables to the monograph~\cite{Defant}.

A~well-known consequence of Bernstein's inequality (see, e.g.,~\cite[Corollary~5.2.3]{QQ}) is that, for all positive integers
$n, m$ there is a~subset $F \subset \mathbb{T}^n$ of cardinality $\text{card}\,F \leq (1+ 20 m)^n$ such that, for every
$P \in \mathcal{T}_{\leq m}(\mathbb{T}^n)$, we have
\begin{equation*}
\label{bern-stein}
\sup_{z \in \mathbb{T}^n} |P(z)| \leq 2 \sup_{z \in F} |P(z)|\,.
\end{equation*}
In other terms, for $N = (1 + 20 m)^n$ the linear mapping
\begin{align}
\label{bernd}
I\colon  \mathcal{T}_{\leq m}(\mathbb{T}^n) \to \ell_\infty^N\,,\,\,\,\,\,I(P) := (P(z))_{z \in F}
\end{align}
is a $2$-embedding of $\mathcal{T}_{\leq m}(\mathbb{T}^n)$ into $\ell^N_\infty$.\\

\medskip

\section{A probabilistic proof}
We use the famous Kahane-Salem-Zygmund inequality (the KSZ-inequality, see e.g., \cite[Theorem~7.11]{Kahane}, \cite{QQ},
or \cite{Defant}, \cite{DefantMastylo}) to improve Theorem~\ref{Thm:KKKK}. Our argument is different from the original
proof of Theorem~\ref{Thm:KKKK}, which we are going to analyze in the next section.

\begin{Theo} \label{KSZinq}
 Let $(\Omega, \mathcal{A}, \mathbb{P})$ be a~probability measure space. Then there is a positive
constant $C$ such that, for each $m,n \in \mathbb{N}$ and for every  trigonometric random polynomial
 $P(\omega, z)= \sum_{\alpha \in T(m,n)} \varepsilon_\alpha(\omega) c_\alpha z^\alpha$,\,
$(\omega, z) \in \Omega \times \mathbb{T}^n$
one has
\[
\int \|P(\omega, \cdot)\|_{\mathcal{T}_{\leq m}(\mathbb{T}^n)}\,d\mathbb{P}(\omega)
\leq C \,\sqrt{n \log (1+m)}\,\|(c_\alpha)\|_{\ell_2(T(m,n))}\,,
\]
where $(\varepsilon_\alpha)_{\alpha\in T(m,n)}$ is a~sequence of Bernoulli variables on
$(\Omega, \mathcal{A}, \mathbb{P})$.
\end{Theo}

\medskip
As we will see in Theorem~\ref{Thm:K}, the following consequence of the KSZ-inequality gives
Kislyakov's multiplier Theorem~\ref{Thm:KKKK} as a~particular case.

\begin{Theo} \label{Thm:SK}
There is a positive constant $C$ such that for all $n,m \in \mathbb{N}$, all Banach spaces
$F=(\mathbb{C}^{T(m,n)},\|\cdot\|)$ and for all sequences  $\xi=(\xi_\alpha)_{\alpha \in T(m,n)}$
\begin{align*}
\sup_{\|\mu\|_{\ell_2(T(m,n))}\leq 1} \|(\mu_{\alpha} \xi_\alpha)_{\alpha \in T(m,n)}\|_F
\leq C \,\sqrt{n \log (1+m)}\,\big\|M_\xi\colon \mathcal{T}_{\leq m}(\mathbb{T}^n)\to
F\big\|\,.
\end{align*}
\end{Theo}

\vspace{1mm}

\begin{proof}
Let $(\varepsilon_\alpha)_{\alpha\in T(m,n)}$ be  a~sequence of Bernoulli variables on some probability space
$(\Omega, \mathcal{A}, \mathbb{P})$, and consider the following three operators:
\begin{align*}
&
R_\xi\colon \ell_2(T(m,n))  \to F\,,\,\, (c_\alpha) \mapsto (c_\alpha \xi_\alpha)
\\[2ex]&
\phi_{K\!SZ}\colon \ell_2(T(m,n))  \to L_1(\mathbb{P}, \mathcal{T}_{\leq m}(\mathbb{T}^n))
\,,\,\,
 (c_\alpha)
 \mapsto
\sum_{\alpha \in T(m,n) } \varepsilon_\alpha(\cdot) c_\alpha z^\alpha
\\&
L_\xi\colon L_1(\mathbb{P}, \mathcal{T}_{\leq m}(\mathbb{T}^n))
\to  F\,,\,\,\,
\sum_{\alpha \in T(m,n)} \varepsilon_\alpha(\cdot) c_\alpha z^\alpha
\mapsto (c_\alpha \xi_\alpha)\,.
\end{align*}
Clearly, $R_\xi = L_\xi \circ \phi_{K\!SZ}$ and
\begin{equation*}
\|R_\xi\| = \sup_{\|\mu\|_{\ell_2(T(m,n))}\leq 1} \|(\mu_{\alpha} \xi_\alpha)_{\alpha \in T(m,n)}\|_F\,.
\end{equation*}
Moreover, by the KSZ-inequality from Theorem~\ref{KSZinq}, we get
\[
\|\phi_{K\!SZ}\| \leq C \sqrt{n \log (1+m)}\,.
\]
We claim that
\[
\|L_\xi\|  \leq \,\big\|M_\xi\colon \mathcal{T}_{\leq m}(\mathbb{T}^n)\to
F\big\|\,.
\]
Indeed, given a~random polynomial  $P\in L_1(\mathbb{P}, \mathcal{T}_{\leq m}(\mathbb{T}^n))$ given by
\[
P(\omega,z)=\sum_{\alpha \in T(m,n)} \varepsilon_\alpha(\omega) c_\alpha z^{\alpha}, \quad\, (\omega, z) \in
\Omega \times  \mathbb{T}^n\,.
\]
Then, for every $\omega \in \Omega$, we have
\[
\|(c_\alpha \xi_\alpha)\|_F  \leq \big\|D_\xi\colon \mathcal{T}_{\leq m}(\mathbb{T}^n)\to F\big\|
\big\| P(\omega, \cdot) \big\|_{\mathcal{T}_{\leq m}(\mathbb{T}^n)}\,.
\]
Now integrating proves the claim.  All together yields
\begin{align*}
\sup_{\|\mu\|_{\ell_2(T(m,n))}\leq 1} & \|(\mu_{\alpha} \xi_\alpha)_{\alpha \in T(m,n)}\|_F \\
& \leq \|L_\xi\|\, \|\phi_{KSZ}\| \leq C \sqrt{n \log (1+m)}\,\big\|M_\xi\colon \mathcal{T}_{\leq m}(\mathbb{T}^n)\to F\big\|\,,
\end{align*}
and so this completes the proof.
\end{proof}

\medskip
Now by Theorem~\ref{Thm:SK} we deduce that Kislyakov's original Theorem~\ref{Thm:KKKK} is a~special
case of the following more general result.

\begin{Coro} \label{Thm:K}
Let $V$ be a closed subspace of $L_\infty(\mathbb{T}^\infty)$.
Then for every   $\ell_1(\mathbb{Z}^{(\mathbb{N})})$-multiplier $\xi = (\xi_\alpha)_{\alpha \in \mathbb{Z}^{(\mathbb{N})}}$ of $V$ we have
\begin{align*}
\sup_{n,m \in \mathbb{ N}}\,\,\,\,\frac{1}{\sqrt{n \log (1+m)}}
\bigg(\sum_{\alpha \in \Gamma \cap T(m,n)} & |\xi_{\alpha}|^2\bigg)^{1/2}
\,\,\,< \infty\,,
\end{align*}
where
\[
 \Gamma = \bigcup_{f \in V} \supp \widehat{f}\subset \mathbb{Z}^{(\mathbb{N})} \,.
\]
In particular,
every $\ell_1(\mathbb{N}_0^{(\mathbb{N})})$-multiplier of $H^\infty(\mathbb{T}^\infty)$
satisfies the preceding estimate whenever  we replace $\Gamma \cap T(m,n)$ by $\Lambda^{\leq}(m,n)$, and
every $\ell_1(\mathbb{N}_0^{(\mathbb{N})})$-multiplier of $C^A(\mathbb{T}^\infty)$ satisfies the
estimate from \eqref{start}\,.
\end{Coro}

\begin{proof}
Let us prove the first statement. For a~fixed $\xi = (\xi_{\alpha})_{\alpha \in \mathbb{Z}^{(\mathbb{N})}}$,  we can
define a~new sequence $\xi' = (\xi_{\alpha}')_{\alpha \in \mathbb{Z}^{(\mathbb{N})}}$ with $\xi'_{\alpha} = \xi_{\alpha}$
if $\alpha \in \Gamma$ and $\xi'_{\alpha} = 0$ otherwise. It is clear that $\xi$ is a $\ell_{1}(\mathbb{Z}^{(\mathbb{N})})$-multiplier
if and only if so does $\xi'$. Thus, applying  Theorem~\ref{Thm:SK} to $\xi'$ we immediately conclude the result. For the second
statement note that
\begin{align*}
\Big(\sum_{\alpha \in \mathbb{N}_0^n \colon \max\{\alpha_1,...,\alpha_n\}\leq m}& |\xi_{\alpha}|^2\Big)^{\frac{1}{2}}
\leq \Big(\sum_{|\alpha| \leq mn} |\xi_{\alpha}|^2\Big)^{\frac{1}{2}} \\
& \leq C \,\sqrt{n \log (1+mn)}\,\,\big\|M_\xi\colon \mathcal{P}_{\leq mn}(\mathbb{T}^n) \to \ell_1(\Lambda^{\leq }(mn,n))\big\| \\
& \leq C \,\sqrt{n \log (1+mn)}\,\,\big\|M_\xi \colon C^A(\mathbb{T}^n) \to \ell_1(\mathbb{N}_0^n)\big\|\,,
\end{align*}
which completes the argument.
\end{proof}

\medskip

Note that Corollary~\ref{Thm:SKS} below recovers the preceding result -- using a different technique of proof with
the advantage of  slightly more explicit  constants.

\medskip

\section{Variants for compact abelian groups}

Based on local Banach space theory and inspired by ideas of \cite{Kisl} from local Banach space theory, we improve  Theorem~\ref{Thm:SK}
and its Corollary~\ref{Thm:K}. Our more abstract approach has the advantage that it allows to consider more general compact abelian groups
instead of only the multidimensional torus. The main result here is Theorem~\ref{Thm:KK} below.

\medskip

We start with the following basic lemma which is a simple consequence of Lemma~\ref{one} and  crucial for our purpose.

\medskip
\begin{Lemm} \label{oneone}
Let $G$ be a compact abelian group, $\Gamma$ a~finite subset in $\widehat{G}$, and $F:=(\mathbb{C}^{\Gamma}, \|\cdot\|)$
a~Banach space. Then for every $\xi = (\xi_\gamma)_{\gamma \in \Gamma}$
\[
\pi_2(M_\xi\colon  C_\Gamma \to F) = \|D_\xi\colon \ell_2(\Gamma) \to F\|
=\sup_{\|\mu\|_{\ell_2(\Gamma)}\leq 1} \|(\mu_{\gamma} \xi_\gamma)_{\gamma \in \Gamma}\|_F\,.
\]
\end{Lemm}
\begin{proof}
The second equality is obvious, and from Lemma~\ref{one} and the orthogonality of the  characters
in $L_2(\nu)$ (where $\nu$ denotes the normalized Haar measure on $G$), we easily deduce  that
\[
\sup_{\|\mu\|_{\ell_2(\Gamma)}\leq 1} \|(\mu_{\gamma} \xi_\gamma)_{\gamma \in \Gamma}\|_F
\leq \pi_2(M_\xi\colon C_\Gamma \to F\,.
\]
To see the reverse estimate, note that for every $f \in C_\Gamma$ one has
\begin{align*}
\|M_\xi f\|_F = \|(\xi_\gamma \widehat{f}(\gamma))\|_F & \leq
\|D_\xi\colon \ell_2(\Gamma) \to F\|\,\|(\widehat{f}(\gamma))\|_{\ell_2(\Gamma)} \\
& = \|D_\xi\colon \ell_2(\Gamma) \to F\|\,
\Big(\int_{G} |\widehat{f}(\gamma)|^2d\nu(\gamma)\Big)^{\frac{1}{2}}\,.
\end{align*}
Then the definition of the $2$-summing norm shows that
\[
\pi_2(M_\xi\colon  C_\Gamma \to F) \leq  \|D_\xi: \ell_2(\Gamma) \to F)\| \,. \qedhere
\]
\end{proof}

\medskip

As a very first application of the preceding lemma we obtain an  interesting reformulation of Theorem~\ref{Thm:SK}.
\medskip

\begin{Coro} \label{coroKSZ}
There is $C >0$ such that for all  $n,m \in \mathbb{N}$, all Banach spaces  $F=(\mathbb{C}^{T(m,n)},\|\cdot\|)$,
 and all   $\xi=(\xi_\alpha)_{\alpha \in T(m,n)}$
\begin{align*}
\pi_2\big(M_\xi\colon \mathcal{T}_{\leq m}(\mathbb{T}^n)\to
F\big)
\leq
C \,\sqrt{n \log(1+ m)}\,\big\|M_\xi\colon \mathcal{T}_{\leq m}(\mathbb{T}^n)\to
F\big\|\,.
\end{align*}
\end{Coro}

\medskip

In view of Bernstein's theorem from \eqref{bernd}, the next result is a strong extension despite the cotype assumption.
It will allow us to get the Kislyakov type multiplier theorems (in particular Theorem~\ref{Thm:KK}) for compact abelian
groups different from the multidimensional torus.

\medskip

\begin{Prop} \label{two}
\label{pi2lem} Let  $I\colon X \hookrightarrow \ell_\infty^N$ be a~$\lambda$-embedding and $F$ a~Banach space. Then for
every operator $T\colon X \to  F$
\[
\pi_2(T) \leq e^{2} \lambda\, C_2(F)\,\sqrt{1+\ln N}\,\|T\|\,.
\]
\end{Prop}

\medskip

\begin{proof}
Define for each $N \in \mathbb{N}$  the $N$-th harmonic number
$
h_N := \sum_{j=1}^N \frac{1}{j}\,,
$
and the discrete measure $\mu_N$ on the power set of $\{1, \ldots, N\}$ given by  $\mu_N(\{j\}) := \frac{1}{j}$.
for each $j\in \{1,\ldots, N\}$. In what follows we need the elementary observations that
$
\log N < h_N \leq 1+\log N,
$
and for  every $\xi = (\xi_i)\in \mathbb{C}^N$
\[
\|\xi\|_{\ell_\infty^N} \leq e \,\|\xi\|_{L_{h_N}(\mu_N)}\,.
\]
Indeed, if $\| \xi\|_{\ell_{\infty}^{N}} = |\xi_{k}|$ for some $k\in \{1, \ldots, N\}$, then
\[
\Big(\sum_{j=1}^{N} \frac{1}{j}|\xi_{j}|^{h_{N}} \Big)^{\frac{1}{h_{N}}} \, \geq \, \frac{1}{k^{1/h_{N}}}\,|\xi_{k}|
= e^{-\frac{\log{k}}{h_{N}}} \, |\xi_{k}| \geq e^{-\frac{\log{k}}{\log{N}}} \, |\xi_{k}| \geq e^{-1} \, |\xi_{k}|\,.
\]
It is well-known that, for every operator $S\colon E\to F$ between Banach spaces and $2 \leq p < \infty$ one has
\[
\pi_2(S) \leq  K_p \,C_2(F)\, \pi_p(S)\,,
\]
where $K_p \leq \sqrt{p}$  is the  best constant from  the right-hand side of Khinchine's inequality for
Rademacher $p$-averages (see \cite[Theorem 5.15]{Pisier}).

Consider the following obvious factorization of an operator $T \colon X \to F$:
\[
T\colon X \stackrel{I} \longrightarrow I(X) \stackrel{I^{-1}} \longrightarrow X  \stackrel{T} \longrightarrow F\,.
\]
Thus the above facts combined with the ideal properties of $p$-summing operators yields
\begin{align*}
\pi_2(T) & \leq \|I\|\,\pi_2(T \circ I^{-1}) \leq \|I\|\,\sqrt{h_N} \,C_2(F)\,\pi_{h_N} (T \circ I^{-1})\\
& \leq \|I\|\,\|I^{-1}\|\,\|T\|\,\sqrt{h_N}\,C_2(F)\,\pi_{h_N} (\text{id}_{I(X)})\\
& \leq  \|T\|\, C_2(F)\,\sqrt{h_N} \,\lambda \,\pi_{h_N} (\text{id}_{\ell_\infty^N})\,.
\end{align*}
Since $\|\text{id}\colon L_{h_N}(\mu_N) \to \ell_\infty^N  \| \leq e$, it follows that
\[
\pi_{h_N} (\text{id}_{\ell^N_\infty}) \leq e \pi_{h_N} (\text{id} \colon \ell^N_\infty \to L_{h_N}(\mu_N))
 = e h_N^{\frac{1}{h_N}} \leq e^2\,. \qedhere
\]
\end{proof}

\medskip

The following theorem is the main result of this section, and we want to mention once again that its proof
is very much inspired by \cite[Theorem~6]{Kisl}.

\medskip

\begin{Theo}
\label{Thm:KK}
Let $G$ be a compact abelian group and  $\Gamma$ a~finite subset in $\widehat{G}$. Assume that
$I\colon C_{\Gamma} \to  \ell_\infty^N$ is a~$\lambda$-embedding, and  $F:=(\mathbb{C}^{\Gamma}, \|\cdot\|)$
a~Banach space. Then, for every $\xi=(\xi_\gamma)_{\gamma \in \Gamma} \in \mathbb{C}^\Gamma$, we have
\[
\sup_{\|\mu \|_{\ell_2(\Gamma)}\leq 1} \|(\mu_\gamma \xi_\gamma )\|_F
\leq  e^2\, \lambda\, C_2(F) \,\sqrt{1+\log N}\,\,
\big\|M_\xi \colon C_\Gamma \to F\big\|\,.
\]
\end{Theo}

\medskip
Comparing with  Theorem~\ref{Thm:SK}, we see that the prize we  pay for the fact that the theorem  applies to
more general groups than the multidimensional  tori, is that the estimate involves the cotype constant of $F$.

\medskip

\begin{proof} By Lemma~\ref{two} and Bernstein's embedding from \eqref{bernd}, we get
\[
\pi_2(M_\xi \colon C_\Gamma \to F) \leq  e^2\, \lambda \, C_{2}(F) \, \sqrt{1+\log N} \, \|M_\xi\colon C_\Gamma \to F\|\,.
\]
Thus, it follows from Lemma~\ref{one} that, for all $f \in C_\Gamma$
\begin{align*}
\|(\widehat{f}(\gamma)\xi_{\gamma})\|_F & \leq  e^2\, \lambda \, C_{2}(F) \, \sqrt{1+\log N} \,
\|M_\xi \colon C_\Gamma \to F\| \Big(\int_G |f(g)|^2\,d\nu \Big)^{\frac{1}{2}} \\
& = e^2\, \lambda \, C_{2}(F) \, \sqrt{1+\log N} \,\|M_\xi \colon C_\Gamma \to F\|
\Big(\sum_{\gamma \in \Gamma} |\widehat{f}(\gamma)|^2\Big)^{\frac{1}{2}}\,.
\end{align*}
Then the conclusion follows by duality.
\end{proof}

\medskip

We finish with the following improvement of Corollary~\ref{Thm:K}; its proof is an immediate consequence of
Theorem~\ref{Thm:KK}, Bernstein's embedding from \eqref{bernd}, the fact that $C_2(\ell_p(\mathbb{Z}^{n}))
\leq \sqrt{2}$ for $1 \leq p \leq 2$ and that by H\"older's inequality, we have
\[
\Big(\sum_{\gamma \in \Gamma} |\xi_{\gamma}|^r\Big)^{1/r} =
\sup_{\|\mu \|_{\ell_2(\Gamma)}\leq 1} \|(\mu_\gamma \xi_\gamma )\|_{\ell_p(\Gamma)}, \quad\, (\xi_\gamma) \in \ell_r(\Gamma)
\]
for all $1 \leq p \leq 2$ with $\frac{1}{r} = \frac{1}{p}-\frac{1}{2}$\,.
\medskip

\begin{Coro} \label{Thm:SKS}
Let $1 \leq p \leq 2$ and $\frac{1}{r} = \frac{1}{p}-\frac{1}{2}$. Then, for each $n,m \in \mathbb{N}$ and for every
$\xi=(\xi_\alpha)_{\alpha \in \mathbb{Z}^n, |\alpha| \leq m}$
\begin{align*}
\Big(\sum_{\alpha \in \mathbb{Z}^n, |\alpha| \leq m} & |\xi_{\alpha}|^r\Big)^{1/r}
\\&
\leq 2\sqrt{2}e^2 \,\sqrt{n \log (1+20m)}\,\,\big\|M_\xi\colon \mathcal{T}_{\leq m}(\mathbb{T}^n)\to
\ell_p\big(\{\alpha \in \mathbb{Z}^n\colon |\alpha| \leq m\}\big)\big\|\,.
\end{align*}
\end{Coro}

\medskip

The following remarks suggest that this corollary leaves some space for improvements.
Given $1 \leq p < \infty$, consider again  the multiplication operator
\[
M_\xi \colon \mathcal{P}_{\leq m}(\mathbb{T}^n) \to \ell_p(\Lambda^{\leq}(m,n)), \quad\,
f \mapsto \big(\widehat{f}(\alpha)\xi_{\alpha}\big)\,.
\]
\noindent
If $2 \leq p < \infty$, then
\[
\sup_{\alpha \in \Lambda^{\leq}(m,n)} |\xi_{\alpha}| = \|M_\xi\|\,,
\]
whereas for  $\frac{2m}{m+1} \leq p \leq 2$ by the hypercontractive BH-inequality \eqref{BH}
\[
\sup_{\alpha \in \Lambda^{\leq}(m,n)} |\xi_{\alpha}| \leq \|M_\xi\|
\leq C^m \sup_{\alpha \in \Lambda^{\leq}(m,n)} |\xi_{\alpha}|\,.
\]
Hence, in view of Corollary~\ref{Thm:SKS} the complex interpolation between the extreme cases
$\ell_1(\Gamma)$ and $\ell_{\frac{2m}{m+1}}(\Gamma)$ suggests the following conjecture.

\medskip

\begin{Conj} \label{doof?}
Let $1 \leq p \leq \frac{2m}{m+1}$ and $\frac{1}{s} =  \frac{m}{m-1}\big(\frac{1}{p}-\frac{m+1}{2m}\big)$.
Then, there exists a~universal~constant $C>0$ such that for each $n,m \in \mathbb{N}$ and for every
$\xi=(\xi_\alpha)_{\alpha \in \Lambda^{\leq}(m,n)}$,

\[
\Big(\sum_{\alpha \in \Lambda^{\leq}(m,n)} |\xi_{\alpha}|^s\Big)^{1/s}
\leq C^{\frac{m}{s}} \,(n \log (1+20m))^{\frac{1}{s}}\,\,
\|M_\xi:\mathcal{P}_{\leq m}(\mathbb{T}^n) \to \ell_p(\Lambda^{\leq}(m,n))\|\,.
\]
\end{Conj}

\medskip

\medskip

\section{Variants  by interpolation}

The purpose of this section is to prove two variants  of Kislyakov's Theorem \ref{Thm:K}, both based on Theorem~\ref{Thm:KK}
and interpolation methods.  We at first recall some notation from interpolation theory (see, e.g., \cite{BL}). The pair
$\vec{X}=(X_0,X_1)$ of Banach spaces is called a~Banach couple if there exists a~Hausdorff topological vector space $\mathcal{X}$
such that $X_j\hookrightarrow\mathcal{X}$, $j=0, 1$. A~mapping $\mathcal{F}$, acting on the class of all Banach couples, is
called  an interpolation functor if for every couple $\xo = (X_0, X_1)$, $\mathcal{F}(\xo)$ is a~Banach space which is intermediate
with  respect to $\xo$, i.e., $X_0\cap X_1 \subset \mathcal{F}(\xo) \subset X_0 + X_1$), and $T\colon \mathcal{F}(\xo) \to \mathcal{F}(\yo)$
is bounded for every operator $T\colon \xo \to \yo$ (meaning $T\colon X_0 + X_1 \to Y_0 + Y_1$ is linear and its restrictions
$T\colon X_j \to Y_j$, $j=0,1$ are defined and bounded). If additionally there is a~constant $C>0$ such that for each
$T\colon \xo \to \yo$
\[
\|T\colon \mathcal{F}(\xo) \to \mathcal{F}(\yo)\| \leq C\,\|T\colon \xo \to \yo\|\,,
\]
where $\|T\colon \xo \to \yo\|:= \max\{\|T\colon X_0 \to Y_0\|, \, \|T\colon X_1 \to Y_1\|\}$, then $\mathcal{F}$ is called
bounded. Clearly, $C\geq 1$, and if $C=1$, then $\mathcal{F}$ is called exact.

For an exact interpolation functor $\mathcal{F}$ we define the fundamental function $\psi_{\mathcal{F}}$
of~$\mathcal{F}$ by
\[
\psi_{\mathcal{F}} (s, t) = \sup \|T\colon \mathcal{F}(\xo) \to \mathcal{F}(\yo)\|,  \quad\, s, t>0\,,
\]
where the supremum is taken over all Banach couples $\xo$, $\yo$ and all operators $T\colon \xo \to \yo$ such that
$\|T\colon X_0 \to Y_0\|\leq s$ and $\|T\colon X_1 \to Y_1\| \leq t$.

\medskip

\begin{Theo} \label{abstract}
Let $G$ be a compact abelian group and $\Gamma$ a~finite subset in $\widehat{G}$.  Suppose that $\mathcal{F}$
is an exact interpolation functor  with the fundamental function $\psi_{\mathcal{F}}$.
Given two norms $\|\cdot\|_1$ and $\|\cdot\|_2$ on $\mathbb{C}^{\Gamma}$ define
\[
F:= \mathcal{F}\big((\mathbb{C}^{\Gamma},\|\cdot\|_1), (\mathbb{C}^{\Gamma},\|\cdot\|_2)\big )\,.
\]
Then, for every $(\xi_\gamma) \in \mathbb{C}^{\Gamma}$, we have
\begin{align*}
\sup_{\|\mu\|_{\ell_2(\Gamma)}\leq 1} &\|(\mu_{\gamma} \xi_\gamma)_{\gamma \in \Gamma}\|_F \\
& \leq \psi_{\mathcal{F}}(1, 1)\,\psi_{\mathcal{F}} \big(\pi_2(M_\xi \colon C_\Gamma \to (\mathbb{C}^{\Gamma},\|\cdot\|_1))\,,\,
\pi_2(M_\xi \colon C_\Gamma \to (\mathbb{C}^{\Gamma},\|\cdot\|_2))\big)\,.
\end{align*}
\end{Theo}

\medskip

 \noindent Note that Lemma~\ref{oneone} shows
\[
\pi_2(M_\xi \colon C_\Gamma \to (\mathbb{C}^{\Gamma},\|\cdot\|_i)) = \|D_\xi\colon \ell_2(\Gamma) \to (\mathbb{C}^{\Gamma},\|\cdot\|_i)\|\,,
\quad i=1, 2\,.
\]
Hence the proof of Theorem~\ref{abstract} is straightforward: It follows from the definition of the function
$\psi_{\mathcal{F}}$ that,  for any operator $T\colon \xo \to \yo$ between the Banach couples $\xo = (X_0, X_1)$ and $\yo=(Y_0, Y_1)$, we have
\[
\big\|T\colon \mathcal{F}(\xo) \to \mathcal{F}(\yo)\big\|
\leq \psi_{\mathcal{F}} \big(\|T\colon X_0 \to Y_0\|, \, \|T\colon X_1\to Y_1\|\big)\,.
\]
In particular this implies that, for any Banach space $Y$,
\[
\|{\rm{id}} \colon Y \to \mathcal{F}(Y, Y)\| \leq \psi_{\mathcal{F}}(1, 1)\,.
\]

\medskip

\subsection{Variant I}

\begin{Theo}
\label{Thm:IntMultB}
Let $G$ be a compact abelian group,  $\Gamma$ a~finite subset in $\widehat{G},$ and $I\colon C_\Gamma \to \ell_\infty^N$
a~ $\lambda$-embedding. Suppose that $\mathcal{F}$ is an exact interpolation functor  with the fundamental
function $\psi_{\mathcal{F}}$, and let
\[
X =\mathcal{F}(\ell_1(\Gamma), \ell_2(\Gamma))\,.
\]
Then, for every $\xi = (\xi_\gamma)_{\gamma \in \Gamma}$, we  have
\begin{align*}
\sup_{\|\mu\|_{\ell_2(\Gamma)}\leq 1} \|(\xi_\gamma \mu_\gamma)_{\gamma \in \Gamma}\|_X
\leq  K \,\,\psi_{\mathcal{F}}\big(\|M_\xi\colon C_\Gamma \to \ell_1(\Gamma)\|, \,\|\xi\|_{\ell_{\infty}(\Gamma)}\big)\,\,
\psi_{\mathcal{F}}\big(\sqrt{1 + \log N}, 1\big) ,
\end{align*}
where $K =  e^2 \lambda \sqrt{2} \,\,\psi_{\mathcal{F}}(1, 1)$.
\end{Theo}

\begin{proof}
Fix $\xi=(\xi_\gamma) \in \mathbb{C}^{\Gamma}$. Since $\ell_1(\Gamma)$ has cotype $2$ with $C_2(\ell_1(\Gamma)\leq \sqrt{2}$,
it follows from  Proposition~\ref{two} that
\[
\pi_2(M_{\xi}\colon C_\Gamma \to \ell_1(\Gamma))
\leq e^2 \sqrt{2} \lambda\,\sqrt{1+\ln N}\,\,\|M_{\xi}\colon C_\Gamma \to \ell_1(\Gamma)\|\,.
\]
Now observe that for any finite sequence $(f_i)_{i=1}^N$ in $C_\Gamma$, we have (where for a~given $g\in G$,
the Dirac functional $\delta_g\in B_{C(G)^{*}}$  is given by $\delta_g(f) :=f(g)$ for all $f\in C(G)$)
\begin{align*}
\sum_{i=1}^N \|M_{\xi}f_i\|_{\ell_2(\Gamma)}^2 & \leq \|\xi\|_\infty \int_{G} \sum_{i=1}^N |f_i(g)|^2\,d\nu \\
& \leq \|\xi\|_{\infty}\,\sup_{g \in G} \sum_{i=1}^N |\delta_g(f_i)|^2 \leq \|\xi\|_\infty\,
\sup_{\|x^{*}\|_{C(G)^{*}} \leq 1} \sum_{i=1}^N |x^{*}(f_i)|^2\,.
\end{align*}
This shows that $\pi_2(M_\xi\colon C_\Gamma \to \ell_2(\Gamma)) \leq \|\xi\|_\infty$. Since $\|\xi\|_\infty =
\|M_\xi \colon C_\Gamma \to \ell_2(\Gamma)\| \leq  \pi_2(M_\xi\colon C_\Gamma \to \ell_2(\Gamma))$, we get
the equality
\begin{align} \label{double}
\pi_2(M_\xi\colon C_\Gamma \to \ell_2(\Gamma)) = \|\xi\|_\infty\,.
\end{align}
Applying Proposition~\ref{abstract}, we conclude (by submultiplicativity of $\psi_{\mathcal{F}}$) that
\begin{align*}
\pi_2(M_\xi \colon C_\Gamma \to  X) & \leq \psi_{\mathcal{F}}(1, 1) \,\,
\psi_{\mathcal{F}}\big(\pi_2(M_\xi\colon  C_\Gamma \to \ell_1(\Gamma)), \pi_2(M_\xi\colon C_\Gamma \to \ell_2(\Gamma)\big) \\
& \leq K\,\, \psi_{\mathcal{F}}\big(\sqrt{1 + \log N}, 1\big) \,\,
\psi_{\mathcal{F}}\big(\|M_\xi\colon C_\Gamma \to \ell_1(\Gamma)\|, \|\xi\|_\infty\big)\,.
\end{align*}
This estimatecombined with Lemma \ref{one} yields the required statement.
\end{proof}

\medskip

In order to see a first consequence, we apply the preceding theorem  to $G=\mathbb{T}^n$ and $C_\Gamma = \mathcal{P}_{\leq m}(\mathbb{T}^n)$
with $\Gamma := \Lambda^{\leq}(m,n)$. For simplicity of notation, we for $1\leq p \leq \infty$ write below $\ell_p$ instead of
$\ell_p(\Lambda^{\leq}(m,n))$.

\medskip

\begin{Coro}
\label{Lambda_estimateA}
Let $\frac{1}{p_{\theta}} = (1-\theta) + \frac{\theta}{2}$ for $\theta\in (0, 1)$. Then for every
$\xi = (\xi_\alpha) \in \mathbb{C}^{\Lambda^{\leq}(m,n)}$, we have
\begin{align*}
\sup_{\|(\mu_\gamma)\|_{\ell_2} \leq 1}
\big\|(\mu_\gamma \xi_\gamma)\big\|_{\ell_{p_\theta}}
\leq
2 e^{2} \sqrt{2}\,\big(\sqrt{n\log(1 + 20 m)}\,
\|M_\xi\colon \mathcal{P}_{\leq m}(\mathbb{T}^n) \to \ell_1\|\big)^{1-\theta} \|\xi\|_{\infty}^{\theta}\,.
\end{align*}
In particular, we have
\begin{align*}
\sup_{\|(\mu_\gamma)\|_{\ell_2} \leq 1}& \big\|(\mu_\gamma \xi_\gamma)\big\|_{\ell_{\frac{2m}{m+1}}} \\
& \leq 2 e^{2} \sqrt{2}\, \big(\sqrt{n(\log(1 + 20m)}\,
\|M_\xi\colon \mathcal{P}_{\leq m}(\mathbb{T}^n) \to \ell_1\|\big)^{\frac{1}{m}} (\|\xi\|_{\infty})^{1- \frac{1}{m}}\,.
\end{align*}
\end{Coro}

\begin{proof}
It is well known that the complex method $\mathcal{F_\theta} := [\,\cdot\,]_{\theta}$ of interpolation has the
fundamental function $\psi_{F}(s, t) = s^{1-\theta}t^{\theta}$ for all $s, t>0$. Applying the well known
interpolation formula for couples of $\ell_p$-spaces, we get
\[
\big[\ell_1, \ell_2\big]_{\theta} \cong \ell_{p_\theta}\,.
\]
Since by Bernstein's result from \eqref{bernd} there exists a~$2$-embedding of $\mathcal{P}_{\leq m}(\T^n)$ into
$\ell_{\infty}^N$ with $N = (1 + 20m)^n$, Theorem~\ref{Thm:IntMultB} gives the first conclusion. To get the second
assertion, we take $\theta = 1- \frac{1}{m}$.
\end{proof}

\medskip

For another seemingly interesting consequence, we recall definition of the abstract Lorentz space
$\Lambda_{\varphi}(\xo)$.

\medskip

For a given function $\varphi \in \mathcal{Q}$ and Banach couple $\xo=(X_0, X_1)$, the abstract Lorentz space
$\Lambda_{\varphi}(\xo)$ is defined to be the space all $x\in X_0 + X_1$ such that
\[
x = \sum_{n\in {\mathbb Z}}\, x_{n}, \quad \text{(convergence in $X_0 + X_1$)},
\]
where $x_{n}\in X_0\cap X_1$, and $\sum_{n\in {\mathbb Z}}\, \varphi(\|x_{n}\|_{X_0}, \|x_{n}\|_{X_1}) <\infty$.
The norm on $\Lambda_{\varphi}(\xo)$ is defined by
\begin{align*} \label{navalny}
\|x\|_{\Lambda_{\varphi}(\xo)} = \inf\, \sum_{n\in {\mathbb Z}}\,
\varphi (\|x_{n}\|_{X_0}, \|x_{n}\|_{X_1}),
\end{align*}
where the infimum is taken over all series described above. It is easily verified that $\Lambda_{\varphi}$ defines
an exact interpolation functor.

\medskip

\begin{Coro}
\label{Lambda_estimate}
Under the notation and assumption of Theorem~\ref{Thm:IntMultB}, we get for $\varphi:=\psi_{\mathcal{F}}$
\begin{align*}
& \big\|{\rm{id}\colon \Lambda_{\varphi}(\mathcal{M}(C_\Gamma, \ell_1(\Gamma)),\ell_\infty(\Gamma)) \to
\mathcal{D}(\ell_2(\Gamma), \mathcal{F}(\ell_1(\Gamma), \ell_2(\Gamma)})\big\|
\leq K \varphi(\sqrt{1 + \log N}, 1).
\end{align*}
\end{Coro}

\begin{proof}
From Theorem \ref{Thm:IntMultB}, for any $\xi=(\xi_\gamma) \in \mathbb{C}^{\Gamma}$, we get that
\begin{align*}
 \|D_{\xi}\colon \ell_2(\Gamma) &\to \mathcal{F}(\ell_1(\Gamma), \ell_2(\Gamma))\|
 \\
& \leq K \,\,\varphi(\sqrt{1 + \log N}, 1)\,\,
\varphi(\|\xi\|_{\mathcal{M}(C_\Gamma, \ell_1(\Gamma))}, \|\xi\|_{\ell_{\infty}(\Gamma)})\,.
\end{align*}
By the construction of the abstract Lorentz space $\Lambda_{\varphi}$
this completes the proof.
\end{proof}

\medskip

\subsection{Variant II}

\begin{Theo}
\label{Thm:IntMult}
Let  $\mathcal{F}$ be an  exact interpolation functor with a~fundamental function $\psi_{\mathcal{F}}$, and for
$m,n \in \mathbb{N}$
\[
X = \mathcal{F}\big(\ell_1(\Lambda^{\leq}(m,n)), \ell_{\frac{2m}{m+1}}(\Lambda^{\leq}(m,n))\big)\,.
\]
Then for every $\xi = (\xi_\alpha)_{\alpha \in\Lambda(m,n)}$, we have
\begin{align*}
\sup_{\|\mu\|_{\ell_2(\Lambda^{\leq}(m,n))}\leq 1} \|(\xi_\alpha \mu_\alpha)_{\alpha \in \Lambda^{\leq}(m, n)}\|_{X}
\leq \,C(m)\,\sqrt{n}\, \|M_{\xi}\colon \mathcal{P}_{\leq m}(\mathbb{T}^n) \to \ell_1(\Lambda^{\leq}(m,n))\|\,,
\end{align*}
where $C(m) = e^2 2 \,\,\sqrt{2}\,\, \psi_{\mathcal{F}}(1, 1 )\, \,\psi_{\mathcal{F}}(\sqrt{\log(1 + 20 m)},1)$.
\end{Theo}

\medskip
In order to prove  this result (see the end of this subsection),  we apply Theorem~\ref{abstract}
in combination with Proposition~\ref{two}
and the following lemma.

\medskip

\begin{Lemm}
\label{surprise}
For each $m, n\in \mathbb{N}$ and for every $\xi \in \mathbb{C}^{\Lambda^{\leq}(m, n)}$,
\[
\pi_2\big(M_\xi \colon \mathcal{P}_{\leq m}(\mathbb{T}^n) \to \ell_{\frac{2m}{m+1}}(\Lambda^{\leq}(m, n)\big)
\leq |\Lambda^{\leq}(m, n)|^{\frac{1}{2m}}\,\|\xi\|_\infty\,.
\]
In particular, if $\xi= \textbf{1}$,
\begin{align*}
\sqrt{1 + \frac{n-1}{m}}  \leq
\pi_2\big(M_\xi \colon \mathcal{P}_{\leq m}(\mathbb{T}^n) &\to \ell_{\frac{2m}{m+1}}(\Lambda^{\leq}(m, n)\big)\\
& = |\Lambda^{\leq}(m, n)|^{\frac{1}{2m}} \leq 2\sqrt{2e}\, \sqrt{1 + \frac{n-1}{m}}\,.
\end{align*}
\end{Lemm}

\begin{proof} We apply Lemma \ref{oneone} in the case $G =\mathbb{T}^n$ with
$\Gamma =\Lambda^{\leq}(m,n) \subset \widehat{G} = \mathbb{Z}^n$ and $F= \ell_{\frac{2m}{m+1}}(\Lambda^{\leq}(m,n))$.
Clearly, for each $m, n\in \mathbb{N}$, we have
\begin{align*}
&
\big\|{\rm{id}} \colon \ell_2(\Lambda^{\leq}(m, n)) \to \ell_{\frac{2m}{m+1}}(\Lambda^{\leq}(m, n))\big\|
\\&
=|\Lambda^{\leq}(m, n)|^{\frac{1}{2m}} =
 \bigg(\sum_{k=0}^m
{k+n-1 \choose k}\bigg)^{\frac{1}{2m}}
\leq
(m+1)^{\frac{1}{2m}} {m+n-1 \choose m}^{\frac{1}{2m}}
\,.
\end{align*}
Applying the well-known quantitative version of Stirling's formula yields
\[
{m+n-1 \choose m} \leq 2 e^m \Big(1 + \frac{n-1}{m}\Big)^m\,.
\]
Now observe that, for any $x\geq y>z >0$, $\frac{x-z}{y-z} \geq \frac{x}{y}$ and whence
\[
{N \choose k} = \frac{N}{k}\,\frac{N-1}{k-1} \cdots \frac{N-k+1}{1}
\geq \Big(\frac{N}{k}\Big)^k, \quad\, 1\leq k\leq N\,.
\]
In consequence, we deduce that
\[
\sqrt{1 + \frac{n-1}{m}} \leq |\Lambda^{\leq}(m, n)|^{\frac{1}{2m}}
\leq 2 \sqrt{2e}\, \sqrt{1 + \frac{n-1}{m}}\,,
\]
and this gives the required estimates.
\end{proof}

\medskip
Finally, we are prepared to give the proof of Theorem~\ref{Thm:IntMult}.

\medskip

\begin{proof}[Proof of Theorem~\ref{Thm:IntMult}]
Note that
\[
\|\xi\|_\infty \leq \beta := \|M_{\xi}\colon \mathcal{P}_{\leq m}(\mathbb{T}^n  \to \ell_1(\Lambda^{\leq}(m,n))\|\,.
\]
Since $\ell_1(\Lambda^{\leq}(m,n)$ has cotype $2$ with $C_2(\ell_1(\Lambda^{\leq}(m,n))\leq \sqrt{2}$,
it follows from Proposition~\ref{two} and the embedding from \eqref{bernd} that
\begin{align*}
\pi_2(M_{\xi}\colon \mathcal{P}_{\leq m}(\mathbb{T}^n) \to \ell_1(\Lambda^{\leq}(m,n)))
\leq e^2 2 \,\sqrt{2}\,
\sqrt{n\log(1 + 20 m)}\, \,\,\beta
\,.
\end{align*}
From Lemma~\ref{surprise} we know that
\[
\pi_2(M_\xi\colon  \mathcal{P}_{\leq m}(\mathbb{T}^n) \to \ell_{\frac{2m}{m+1}}(\Lambda^{\leq}(m, n)) \leq 2\sqrt{2e} \sqrt{n} \,\beta\,.
\]
Applying Theorem~\ref{abstract}, we conclude (by submulitiplicativity of $\psi_{\mathcal{F}}$)
that with $C=\psi_{\mathcal{F}}(1, 1)$ that
\begin{align*}
& \sup_{\|\mu\|_{\ell_2(\Lambda^{\leq}(m,n))}\leq 1} \|(\mu_{\gamma} \xi_\gamma)_{\gamma \in \Gamma}\|_F \\ & \leq C
\psi_{\mathcal{F}}(\pi_2(M_\xi\colon  \mathcal{P}_{\leq m}(\mathbb{T}^n) \to \ell_1(\Lambda^{\leq}(m,n))), \pi_2(M_\xi\colon \mathcal{P}_{\leq m}(\mathbb{T}^n)\to \ell_2(\Lambda^{\leq}(m,n))) \\
& \leq C \psi_{\mathcal{F}}(e^2 2 \,\sqrt{2}\, \sqrt{n\log(1 + 20 m)}\, \beta, 2 \sqrt{2e} \sqrt{n} \,\beta )\\
& \leq C e^2 2 \,\sqrt{2} \, \sqrt{n}\,\psi_{\mathcal{F}}(\sqrt{\log(1 + 20 m)},1)
\,\beta \,. \,\,\,\,\,\,\qedhere
\end{align*}
\end{proof}

\medskip
The following corollary 'interpolates' the estimates  from Corollary~\ref{Thm:SK} ($p=1$) and Lemma~\ref{surprise} ($p=2m/m+1$)\,.
\medskip

\begin{Coro} \label{Thm:SKSK}
For  $m \in \mathbb{N}$ let $1 \leq p \leq \frac{2m}{m+1}$. Then, for each $n\in \mathbb{N}$ and every
$\xi=(\xi_\alpha)_{\alpha \in \Lambda^{\leq}(m,n)}$,
\[
\Big(\sum_{\alpha \in \Lambda^{\leq}(m,n)} |\xi_{\alpha}|^r\Big)^{1/r} \leq 2\sqrt{2}e^2 \,
\sqrt{n}  \sqrt{\log (1+20m)}^{\beta_m}\,\,\|M_\xi\colon \mathcal{P}_{\leq m}(\mathbb{T}^n)\to \ell_p(\Lambda^{\leq }(m,n))\|\,,
\]
where $\frac{1}{r} = \frac{1}{p} -\frac{1}{2}$ and
$
\beta_m = 1 - \frac{1-1/p}{1-\frac{m+1}{2m}}\,.
$
\end{Coro}

\begin{proof}
Define $\theta_m \in (0,1)$ by $\frac{1}{p} = \frac{1-\theta_m}{1} +\frac{\theta_m}{\frac{2m}{m+1}}\,.$
Then $1-\theta_m = \beta_m$. Hence by Theorem~\ref{Thm:IntMult} and as in the proof of Corollary~\ref{Lambda_estimateA},
it follows that for every $\xi = (\xi_\alpha)_{\alpha \in\Lambda^{\leq}(m,n)}$, we get
\begin{align*}
\Big(\sum_{\alpha \in \Lambda^{\leq}(m,n)} |\xi_{\alpha}|^r\Big)^{1/r} \leq \,C(m)\,\sqrt{n}\,
\|M_{\xi}\colon \mathcal{P}_{\leq m}(\mathbb{T}^n) \to \ell_1(\Lambda^{\leq}(m,n))\| \,,
\end{align*}
where $C(m) = 2 \sqrt{2}e^2 \,(\log(1 + 20 m))^{\beta_m}$.
\end{proof}

\medskip

\section{Applications}

\bigskip

\medskip

\subsection{Multipliers of analytic trigonometric polynonials}
Recall from  \eqref{P1} the definition of $\mathcal{P}_{\leq m}(\mathbb{T}^n)$, all  trigonometric analytic
polynomials of degree $\leq m$ on the $n$-dimensional torus $\mathbb{T}^n$, and its subspace
$\mathcal{P}_{=m}(\mathbb{T}^n)$ of all $m$-homogeneous polynomials.

We also recall some definition, which we already touched in the introduction, the Hardy space
\begin{equation*}
H_{\infty}(\mathbb{T}^{\infty}) := \Big\{ f \in  L_{\infty} (\mathbb{T}^{\infty}); \,\, \widehat{f}(\alpha) = 0
\,  , \, \,\, \, \forall \alpha \in \mathbb{Z}^{(\mathbb{N})} \setminus \mathbb{N}_{0}^{(\mathbb{N})} \Big\}\, ,
\end{equation*}
and its $m$-homogeneous part
\begin{equation*} \label{Hpm}
H_{\infty}^{m}(\mathbb{T}^{\infty}) := \Big\{ f \in H_{\infty} (\mathbb{T}^{\infty}); \,\,
\widehat{f} (\alpha) \neq 0 \, \Rightarrow \, \vert \alpha \vert = m \Big\}\,.
\end{equation*}

\medskip

\subsubsection{Sidon constants I} \label{sid1}

Let $1 \leq p < \infty$. By $\chi_p\big( \mathcal{P}_{\leq m}(\mathbb{T}^n) \big)$ we denote the $p$-Sidon constant
of $\mathcal{P}_{\leq m}(\mathbb{T}^n) $, that is, the best constant $c >0$ such that,  for all polynomials
$f(z)= \sum_{\alpha \in \Lambda^{\leq } (m,n)} \widehat{f}(\alpha)z^\alpha$, we have
\[
\Big(\sum_{\alpha \in \Lambda^{\leq } (m,n)} |\widehat{f}(\alpha)|^p\Big)^{\frac{1}{p}} \leq c\, \|f\|_\infty\,.
\]
Similarly we define $\chi_p\big( \mathcal{P}_{= m}(\mathbb{T}^n) \big)$,
but in this case  we only consider $m$-homogeneous trigonometric polynomials instead of all
trigonometric polynomials of degree  less than or equal $m$. Since for every
$f\in \mathcal{P}_{\leq m}(\mathbb{T}^n)$
\[
\Big(\sum_{|\alpha| \leq m} |\widehat{f}(\alpha)|^2\Big)^{\frac{1}{2}}
=\Big(\int_{\mathbb{T}^n} |f(z)|^2\,dz\Big)^{\frac{1}{2}} \leq  \|f\|_\infty\,,
\]
it follows that, for each $m,n \in \mathbb{N}$,
\[
\chi_p\big( \mathcal{P}_{= m}(\mathbb{T}^n) \big)
= \chi_p\big( \mathcal{P}_{\leq m}(\mathbb{T}^n) \big) =1, \quad\, 2\leq p \leq \infty\,.
\]
Observe also that the Bohnenblust--Hille inequality \eqref{BH} combined with H\"older's inequality imply
that there exists a~constant $C>0$ such that, for each $m,n \in \mathbb{N}$, we have
\begin{equation} \label{sido}
1\leq \chi_p\big( \mathcal{P}_{= m}(\mathbb{T}^n) \big)\leq \chi_p\big( \mathcal{P}_{\leq m}(\mathbb{T}^n) \big) \leq C^{m},
\quad\, \frac{2m}{m+1} \leq p < 2\,.
\end{equation}

\medskip

\begin{Theo} \label{app1}
Let $1 \leq p \leq \frac{2m}{m+1}$ and $\frac{1}{r}=\frac{1}{p} - \frac{1}{2}$. Then there is
a~universal positive constant $\gamma$ such that, for each $m,n \in \mathbb{N}$
\[
\frac{1}{\gamma^m}\Big( \frac{n}{m}\Big)^{\frac{m}{r}-\frac{1}{2}} \leq \chi_p\big( \mathcal{P}_{=m}(\mathbb{T}^n) \big)
\leq \chi_p\big( \mathcal{P}_{\leq m}(\mathbb{T}^n) \big) \leq \gamma^m \Big( \frac{n}{m}\Big)^{\frac{m}{r}-\frac{1}{2}}\,.
\]
\end{Theo}
\bigskip

Note that  $\frac{m}{r}-\frac{1}{2} \ge 0$ whenever $1 \leq p \leq \frac{2m}{m+1}$ and $\frac{1}{r}=\frac{1}{p} - \frac{1}{2}$.
Moreover, as it should be, we have that $\frac{m}{r}-\frac{1}{2}=\frac{m-1}{2}$ for $p=1$ and $\frac{m}{r}-\frac{1}{2}= 0$ for
$p=\frac{2m}{m+1}$.

\medskip

For the homogeneous case and $p=1$ this result is proved  in \cite{DefantFrerickOrtegaOunaiesSeip} (see also \cite[Theorem 9.10]{Defant})
-- the upper estimate is based on the hypercontractivity of the Bohnenblust-Hille inequality, and the lower estimate on the
Kahane--Salem--Zygmund inequality. Here we deduce the upper inequality from the case $p=1$  by applying the complex interpolation method, and the lower estimate by the following independently interesting lemma based on
Corollary~\ref{Thm:SKS}.

\begin{Lemm} \label{unc}
Let $1 \leq p \leq \frac{2m}{m+1}$ and $z \in  \mathbb{C}^n$, and denote by $z^\ast$ the decreasing
rearrangement of $z$. Then,  for each $m \in \mathbb{N}$, we have

\begin{equation*} \label{crucial2}
(z^\ast_n)^m \leq 2\sqrt{2}e^2  \sqrt[r]{m!}\, \sqrt{\log (1+20m)} \,
\frac{1}{n^{\frac{m}{r}-\frac{1}{2}}}\,\big\|M_{z}\colon \mathcal{P}_{=m}(\mathbb{T}^n) \to \ell_p(\Lambda^{=}(m,n))\big\|\,,
\end{equation*}
where  $\frac{1}{r}=\frac{1}{p} - \frac{1}{2}$ and
\[
M_z \colon \mathcal{P}_{=m}(\mathbb{T}^n) \rightarrow \ell_{p}(\Lambda^{=}(m,n))\,,\,\, f \to \big(\widehat{f}(\alpha)(z^\ast)^\alpha\big)\,.
\]
\end{Lemm}

\begin{proof}
From  Corollary~\ref{Thm:SKS}, we deduce that
\begin{align*}
(z^\ast_n)^{rm} \sum_{|\alpha |=m}1 & = \sum_{ |\alpha |=m} \big( (z_n^\ast)^{\alpha_1} \cdots (z_n^\ast))^{\alpha_n} \big)^r \\
& \leq \sum_{ |\alpha |=m} (z^\ast)^{r \alpha} \leq (2\sqrt{2}e^2)^r\, (n  \log (1+20m))^{\frac{r}{2}}\,\|M_z\|^r\,.
\end{align*}
Since $\operatorname{dim}\mathcal{P}_m(\mathbb{T}^n) = \binom{n+m-1}{m}$, we get
\[
(z^\ast_n)^{rm} \frac{n^m}{m!}\leq (z^\ast_n)^{rm}  \binom{n+m-1}{m}
\leq  (2\sqrt{2}e^2)^r\,(n  \log (1+20m))^{\frac{r}{2}}\, \|M_z\|^r\,,
\]
and the desired result follows by taking roots.
\end{proof}

\medskip
We are ready to give the

\begin{proof}[Proof of Theorem~\ref{app1}]  Lower bound:
If $\pmb{1} = (1, \ldots ,1) \in \mathbb{C}^{\Lambda^{=}(m,n)}$, then by the definition we obtain
\[
\|M_{\pmb{1}}\colon \mathcal{P}_{=m}(\mathbb{T}^n) \rightarrow \ell_{p}(\Lambda^{=}(m,n))\|
=\chi_p\big( \mathcal{P}_{= m}(\mathbb{T}^n) \big)\,.
\]
Then by  Lemma~\ref{unc}
\[
1 \leq  2\sqrt{2}e^2  \sqrt[r]{m!}\, \sqrt{\log (1+20m)}\,
\chi_p\big( \mathcal{P}_{= m}(\mathbb{T}^n) \big) \frac{1}{n^{\frac{m}{r}-\frac{1}{2}}}\,,
\]
and so
\begin{align*}
\frac{1}{\gamma^m}\Big( \frac{n}{m}\Big)^{\frac{m}{r}-\frac{1}{2}} & \leq
\frac{1}{2\sqrt{2}e^2m^{\frac{1}{2}} (1 + \log m)^{\frac{1}{2}}}\Big(\frac{n}{m}\Big)^{\frac{m}{r}-\frac{1}{2}}\\
& \leq \frac{m^{\frac{m}{r}-\frac{1}{2}}}{2\sqrt{2}e^2(m!)^{\frac{1}{r}} (1 + \log m)^{\frac{1}{2}}}\Big(\frac{n}{m}\Big)^{\frac{m}{r}-\frac{1}{2}}
\leq \chi_p\big( \mathcal{P}_{= m}(\mathbb{T}^n) \big)\,.
\end{align*}

\noindent Upper bound: Define $0 \leq \theta \leq 1$ by
\[
\frac{1}{p} = \frac{\theta}{1} + \frac{1-\theta}{\frac{2m}{m+1}}\,,
\]
then
\[
\frac{m-1}{2} \theta =  \frac{m}{r}-\frac{1}{2}\,.
\]
We know that
\[
\|M_{\pmb{1}}\colon \mathcal{P}_{\leq m}(\mathbb{T}^n) \rightarrow \ell_1(\Lambda^{\leq}(m,n))\| \leq C_1^m\Big(\frac{n}{m}\Big)^{\frac{m-1}{2}}
\]
and by \eqref{sido}
\[
\|M_{\pmb{1}}\colon \mathcal{P}_{\leq m}(\mathbb{T}^n) \rightarrow \ell_{\frac{2m}{m+1}}(\Lambda^{\leq}(m,n))\| \leq C_2^m\,.
\]
 Applying the complex interpolation method, we get
\[
\|M_{\pmb{1}}\colon \mathcal{P}_{\leq m}(\mathbb{T}^n) \rightarrow \ell_{p}(\Lambda^{\leq}(m,n))\|
\leq \Big(C_1^m\Big(\frac{n}{m}\Big)^{\frac{m-1}{2}}  \Big)^\theta \Big( C_2^m  \Big)^{1-\theta}
\leq \gamma^m \Big( \frac{n}{m}\Big)^{\frac{m}{r}-\frac{1}{2}}\,,
\]
and so the conclusion follows.
\end{proof}

\medskip

Why would a positive answer to Conjecture~\ref{doof?} not lead to a better lower bound?
In this case we  for $\frac{1}{s} =  \frac{m}{m-1}\big(\frac{1}{p}-\frac{m+1}{2m}\big)$ would get that
\[
\frac{1}{\gamma^m}  \Big(  \frac{n}{m} \Big)^{\frac{m}{s}-\frac{1}{s}} \leq \chi_p\big( \mathcal{P}_{= m}(\mathbb{T}^n) \big)\,.
\]
But $\frac{m}{s}-\frac{1}{s} = \frac{m}{r}-\frac{1}{2}$, where again  $\frac{1}{r}= \frac{1}{p}-\frac{1}{2}$, and so
we would not arrive at a~contradiction.

\medskip

\subsubsection{Bohr radii}

Denote by $K_n$ the $n$th Bohr radius, that is the best $0 < r \leq 1$ such that, for every $f\in  H_{\infty}(\mathbb{T}^{\infty})$,
we have
\[
\sum_{\alpha \in \mathbb{N}_0^{(\mathbb{N})}} |\widehat{f}(\alpha)|\,r^{|\alpha|} \leq \|f\|_\infty\,.
\]
It is known  that
\begin{equation}\label{Bohr}
\lim_{n \to \infty} \frac{K_n}{\sqrt{\frac{\log n}{ n}}} =1\,;
\end{equation}
this was established in \cite{Bohrradius} extending an earlier result of \cite{DefantFrerickOrtegaOunaiesSeip}, which basically proved
that the limit is between $1$ and $\sqrt{2}$. For a detailed account on all this, see  the monograph \cite{Defant}.

The original proof of the lower estimate in \eqref{Bohr} is  based on the Kahane--Salem--Zygmund inequality (see, e.g.,
\cite[Theorem~7.1]{Defant}). Let us indicate an alternative argument based on Theorem~\ref{Thm:KK}. We have that for
each $m\in \mathbb{N}$
\[
K_n \leq K^m_n\,,
\]
where $K^m_n$ is defined like $K_n$ only  taking into account functions from $H_m(\mathbb{T}^\infty)$ instead of all functions from $H_m(\mathbb{T}^\infty)$. Then a simple reformulation shows that
\[
K_n \leq K^m_n = \frac{1}{ \sqrt[m]{\chi(m,n)}}\,
\]
(see \cite[(9.15)]{Defant}). Using the lower estimate from Theorem~\ref{app1} (which was proved with Kislyakov's ideas), and copying
word by word the proof given in \cite[Theorem~8.22]{Defant}, we obtain an alternative approach to the  upper  bound in \eqref{Bohr}
 without using the Kahane--Salem--Zygmund inequality.

\medskip

\subsubsection{Monomial convergence I}

Let $V$ be a subset of $H_\infty(\mathbb{T}^\infty)$, then
\[
\mon V := \bigg\{ z \in \mathbb{C}^{\mathbb{N}};\, \, \sum_{\alpha \in \mathbb{N}_0^{(\mathbb{N})}}
\vert \widehat{f}(\alpha) z^{\alpha} \vert < \infty \,\,\,\text{for all }  f\in  V \bigg\}
\]
is called the set of monomial convergence of $V$. If $V$ is a closed subspace of $H_\infty(\mathbb{T}^n)$,
then a~simple closed graph argument shows that $z \in  \mon V$ if and only if there is $C = C(z) >0$ such that,
for every $f \in  V$,
\[
\sum_{\alpha \in \mathbb{N}_0^{(\mathbb{N})}}  |\widehat{f}(\alpha) z^\alpha| \leq C \|f\|_\infty\,.
\]
A sequence  $\xi = (\xi_\alpha)_{\alpha \in\mathbb{N}_0^{(\mathbb{N})}}$ is said to be multiplicative whenever
for all $\alpha, \beta \in \mathbb{N}_0^{(\mathbb{N})}$, we have $\xi_{\alpha + \beta} = \xi_{\alpha} \xi_{\beta} $\,.
Equivalently, $\xi$ is multiplicative if and only if there is $z \in \mathbb{C}^{\mathbb{N}}$ such that, for all
$\alpha \in \mathbb{N}_0^{(\mathbb{N})}$, we have $\xi_\alpha = z^{\alpha}$, for all $\alpha \in \mathbb{N}_0^{(\mathbb{N})}$
(if $\xi$ is multiplicative, then define $z_k = \xi_{e_k}$ for each $k \in \mathbb{N}$)\,.

\medskip

\begin{Rema}
Let $V \subset H_\infty(\mathbb{T}^\infty)$ be a closed subspace. Then $\mon V$ equals the set of all  multiplicative
$\ell_1(\Gamma)$-multipliers $\xi = (\xi_\alpha)_{\alpha \in\mathbb{N}_0^{(\mathbb{N})}}$ for $V$, where
\[
\Gamma :=\bigcup_{f \in V} \supp \widehat{f} \subset \mathbb{Z}^{(\mathbb{N})}\,.
\]
\end{Rema}

We refer to the monograph \cite{Defant} for a~detailed exposition on the sets of monomial convergence of
$H_\infty(\mathbb{T}^\infty)$ and its closed  subspace $H_\infty^m(\mathbb{T}^\infty)$ (together with all
its consequences for spaces of holomorphic functions in infinitely many variables and ordinary Dirichlet series).
In particular, the following  results  from  \cite{BayartDefantFrerickMaestreSevilla} (see also
\cite[Theorem~10.1 and  Theorem~10.15]{Defant}) give two (almost) complete description of both sets.

\medskip

\begin{Theo} \label{MON}\text{}
\begin{itemize}
\item[{\rm(i)}] $ \mon H_\infty^m(\mathbb{T}^\infty) = \ell_{\frac{2m}{m-1}, \infty}$ for each $m\in \mathbb{N}$\,{\rm;}
\item[{\rm(ii)}] For every $z \in \mathbb{C}^{\mathbb{N}}$ the following two statements hold\,{\rm:}
\begin{enumerate}
\item[{\rm(a)}] If \,$\displaystyle\limsup_{n \to \infty} \frac{1}{\log n} \sum_{j=1}^{n} z^{* 2}_{j} < 1$,
then\, $z \in  \mon H_\infty(\mathbb{T}^\infty)$\,.
\item[{\rm(b)}] If  $z \in  \mon H_\infty(\mathbb{T}^\infty)$, then\,\,
$\displaystyle\limsup_{n \to \infty} \frac{1}{\log n} \sum_{j=1}^{n} z^{*2}_{j} \leq 1$. Moreover, here the converse
implication is false.
\end{enumerate}
\end{itemize}
\end{Theo}

\medskip

In fact the 'lower inclusions' for $\mon H_\infty^m(\mathbb{T}^\infty)$ and $\mon H_\infty^m(\mathbb{T}^\infty)$  follow
from the Kahane--Salem--Zygmund theorem. Alternatively, the following two proof show that these lower inclusions also may
be deduced from  Theorem~\ref{Thm:KK}.

\begin{proof}[Alternative proof of $\ell_{\frac{2m}{m-1}} \subset \mon H_\infty^m(\mathbb{T}^\infty)$ in Theorem~$\ref{MON}$]
Fix  $z \in \mon H_\infty(\mathbb{T}^\infty)$, and recall  that  then the decreasing rearrangement
$z^\ast \in \mon H_\infty(\mathbb{T}^\infty)$ (see \cite[Remark 10.4]{Defant}). Then
\[
M_{z^\ast}\colon H_\infty(\mathbb{T}^\infty) \rightarrow \ell_1({\mathbb{N}_0^{(\mathbb{N})})}\,,\,\,
f \to \big(\widehat{f}(\alpha)(z^\ast)^\alpha\big)
\]
is bounded. Consequently, by Lemma~\ref{unc} there is some universal constant $\gamma >0$ such that, for each $m\in \mathbb{N}$,
\[
(z^\ast_n)^m \leq \gamma  \sqrt{m!}\, \log m \, \frac{1}{n^{\frac{m-1}{2}}}\,\|M_{z^\ast }\|\,.
\]
Taking the $m$th-rot, we conclude that
\[
z^\ast_n \leq \gamma^\frac{1}{m}  \sqrt{m!}^\frac{1}{m}\, (\log m)^\frac{1}{m}\,
\|M_{z^\ast }\|^\frac{1}{m} \frac{1}{n^{\frac{m-1}{2m}}} \ll_m \frac{1}{n^\frac{m-1}{2m}}\,,
\]
and this completes the proof.
\end{proof}

\medskip

\begin{proof}[Alternative proof of ${\rm(2b)}$ in Theorem~\ref{MON}] With the closed graph argument from the preceding
proof and Theorem~\ref{Thm:SK}, we see that for some constant $\gamma >0$ and each $m,n\in \mathbb{N}$, we have
\[
\Big(\sum_{\substack{\alpha \in \N_0^{n}\\ |\alpha |=m}} |(z^\ast)^{\alpha}|^2\Big)^{\frac{1}{2}}
\leq \gamma\,\sqrt{ n \log m}\,.
\]
Then the proof finishes exactly as in \cite[Section 10.5.1]{Defant}.
\end{proof}

\medskip

\subsubsection{Bohr-Bohnenblust-Hille theorem}
Referring to a couple of results, which were worked out in  the recent monograph \cite{Defant}, we intend to show that our
alternative approach to Theorem~\ref{MON}.$(ii.b)$ leads to an alternative solution of Bohr's famous absolute convergence
problem on ordinary Dirichlet series.

This problem  asked for the largest possible width $S$ of the strip in the complex plane on which an ordinary
Dirichlet series $D=\sum a_n n^{-s}$ converges uniformly but not absolute (see \cite[Section 1]{Defant}).

Bohr himself established the upper estimate $S \leq~1/2$ (see \cite[Proposition 1.10]{Defant}), but he was not able to decide
whether this upper bound is optimal. A~non-trivial reformulation (still due to Bohr, see \cite[Proposition~1.24]{Defant})
shows that
\begin{equation} \label{S}
S = \sup_{D \in \mathcal{H}_{\infty}} \sigma_a(D)\,,
\end{equation}
where  $\mathcal{H}_\infty$ denotes the Banach space of all Dirichlet series $D$, which on the positive half plane converge
pointwise to a bounded (and then necessarily holomorphic) function, and $\sigma_a(D) \in \mathbb{R}$ defines  the abscissa
of absolute convergence of $D$.

On the other hand,
\[
\mathcal{H}_\infty =  H_\infty(\mathbb{T}^\infty)\,,
\]
i.e., there is an isometric linear bijection between both Banach spaces  preserving  Dirichlet and Fourier coefficients.
More precisely, if this bijection identifies \mbox{$f\in H_{\infty}(\mathbb{T}^{\infty})$} and $D = \sum a_n n^{-s} \in \mathcal{H}_\infty$,
then  $a_n = \widehat{f}(\alpha)$, whenever $n = \mathfrak{p}^\alpha$ (here $\mathfrak{p} = (p_n)$ stands for the sequence
of primes (see \cite[Corollary 5.3]{Defant}). Then this fact combined with \eqref{S} shows that
\[
S = \inf \Big\{ \sigma >0;\,\,\frac{1}{\mathfrak{p}^{\sigma}} \in \mon H_\infty(\T^\infty) \Big\}
\]
(see \cite[(10.5)]{Defant}). As a~consequence  we see that Theorem~\ref{MON}.$(iib)$, for which we gave an alternative
proof using  Kislyakov's ideas condensed in Theorem~\ref{Thm:KK}, immediately proves that $S\ge 1/2$.

So all in all, we arrive at a~new proof of the so-called Bohr-Bohnenblust-Hille Theorem: $S= \frac{1}{2}$, which in fact
is in the very center of the discussion in the monograph \cite{Defant} (for the highly non-trivial original proof due to
Bohnenblust and Hille from 1931 see  \cite[Section 2]{Defant}. We finally remark that this  monograph also contains a~couple
of other  proof -- none of them being trivial.

\medskip

\subsection{Multipliers of functions on  Boolean cubes}

For $N \in \mathbb{N}$ let $\mathcal{B}_N$ be the set of all functions $f:\{-1,1\}^{N} \rightarrow \R$.
Recall for $f\in \mathcal{B}_N$ the expectation is given by
\[
\mathbb{E}\big[ f \big] := \frac{1}{2^{N}} \sum_{x \in \{-1,1\}^{N}}{f(x)}\,.
\]
The dual group of $\{-1,1\}^{N}$ actually consists of the set of all Walsh functions $\chi_{S}$ for $S \subset [N]$, which allows
to  associate to each such $f \in \mathcal{B}_N$ its Fourier-Walsh expansion
\begin{equation}\label{equa:FourierExpansionShort}
f(x) = \sum_{S \subset [N]}{\widehat{f}(S) \, x^S}\, , \hspace{3mm} x \in \{-1, 1\}^{N}\,,
\end{equation}
where $x^{S}:= \chi_{S}(x):= \prod_{n \in S}x_{n}$ are the Walsh functions, and the coefficients are given by
$\widehat{f}(S) = \mathbb{E}[ f \chi_{S} ]$. Thereby, a~nonzero function $f$ of degree $d$ satisfies that $\widehat{f}(S) = 0$
provided $|S| > d$. We say that $f$ is $m$-homogeneous whenever $\widehat{f}(S) = 0$ provided $|S| \neq m$.

Given $m,d \in \mathbb{N}$  with $m,d \leq N$, we write $\mathcal{B}_N^{=m}$ for all $m$-homogeneous functions in $\mathcal{B}_N$,
and $\mathcal{B}_N^{\leq d}$ for all  functions of degree $\leq d$. Moreover, we define
\[
\mathcal{B} := \bigcup_N \mathcal{B}_N,\,\,\, \, \mathcal{B}^{= m} := \bigcup_N \mathcal{B} ^{= m}_N, \,\,\, \text{and}\,\,\,
\mathcal{B}^{\leq d} := \bigcup_N \mathcal{B} ^{\leq d}_N\,.
\]
Similar to \eqref{BH}, we have the following 'hypercontractive' Bohnenblust-Hille inequality for functions on the Boolean cube from \cite{DefantMastyloPerezB}: There is a~universal constant $C >1$ such that, for each $d, N\in \mathbb{N}$ and for every
$f\in \mathcal{B}_N^{\leq d}$, we have
\begin{align} \label{BHBoole}
\Big(\sum_{S \subset [N]\colon |S|\leq d} |\widehat{f}(S)|^{\frac{2m}{m+1}}\Big)^{\frac{m+1}{2m}} \leq C^{\sqrt{m \log m}} \|f\|_\infty\,.
\end{align}

\medskip
\subsubsection{Multipliers}

We start with variants of Theorem~\ref{Thm:KK} and Corollary~\ref{Thm:SK} for functions on the Boolean cubes.

\medskip

\begin{Theo} \label{MietekM} Assume that $1 \leq p \leq 2$ and $d,N \in \mathbb{N}$ with $d \leq N$. Define $1 \leq r < \infty$ by $\frac{1}{r}=\frac{1}{p}-\frac{1}{2}$
\begin{itemize}
\item[{\rm(i)}]
For every $ \xi=(\xi_S)_{S \subset [N]}$
\begin{align*}
\frac{1}{\sqrt{1 +N \log 2}}\Big(\sum_{S \subset [N]} |\xi_S|^r\Big)^{\frac{1}{r}} \leq
2\sqrt{2} e^2 \big\|M_\xi\colon \mathcal{B}_N \to \ell_p(\{S \colon S \subset [N] \})\big\|\,.
\end{align*}
\item[{\rm(ii)}] For every  $\xi=(\xi_S)_{S \subset [N], |S|\leq d}$
\begin{align*}
\frac{1}{\sqrt{1+N\log(1+20 d)}}&\Big(\sum_{_{S \subset [N], |S|\leq d}} |\xi_S|^r\Big)^{\frac{1}{r}} \\
& \leq 2\sqrt{2} e^2 (1 +\sqrt{2})^d \big\|M_\xi\colon  \mathcal{B}^{\leq d}_N \to \ell_p(\{S \colon |S|\leq d \})\big\|\,.
\end{align*}
Moreover, in the homogeneous case  $\xi=(\xi_S)_{S \subset [N], |S|= d}$, we may replace the constant on the right side
by $2\sqrt{2} e^2 2^{d-1}$.
\end{itemize}
\end{Theo}

\begin{proof}
For the proof of both statements we  use the fact that $C_2(\ell_p(\Gamma)) \leq \sqrt{2}$ for $1 \leq p \leq 2$. Consider
then for the proof of ${\rm(i)}$ the  canonical isometric embedding
\[
\mathcal{B}_N \ni f \mapsto (f(x))_{x \in \{-1,1\}^N} \in \ell_\infty^{2^N}\,.
\]
Then the conclusion is immediate from Theorem~\ref{Thm:KK}. The proof of ${\rm(ii)}$ is slightly more involved: Denote by
$\mathcal{P}_{\leq d}([-1,1])^N)$ the space of all real polynomials
$f(x)=\sum_{\alpha\in \mathbb{N}_0^n: |\alpha|\leq d} c_\alpha x^\alpha, \, x \in \mathbb{R}^N$, and endow it with the
supremum norm on the $N$-dimensional cube $[0,1]^N$. Since every $f \in \mathcal{B}_N$ can be viewed as a~tetrahedral
polynomial in $ \mathcal{P}_{\leq d}([-1,1])^N) $ with equal norm, the canonical embedding
\[
\mathcal{B}^{\leq d} \longrightarrow \mathcal{P}_{\leq d}([-1,1])^N)
\]
is isometric. Next we look at the canonical embedding
\[
\mathcal{P}_{\leq d}([-1,1])^N) \longrightarrow  \mathcal{P}_{\leq d}(\mathbb{T}^N)\,,
\]
which by a~result of  Klimek \cite{K} is an $(1 + \sqrt{2})^d$ -embedding, and finally we recall that by  \eqref{bernd}
there is a~$2$-embedding
\[
I\colon  \mathcal{P}_{\leq d}(\mathbb{T}^N) \to \ell_{\infty}^M\,,
\]
where  $M = (1 + 20 m)^N$. Then the conclusion again follows from Theorem~\ref{Thm:KK}. In the $d$-homogeneous case we
replace Klimek's by a~result of Visser \cite{V}, which states that the canonical map  from $\mathcal{P}_{= d}([-1,1])^N)$
into $\mathcal{P}_{=d}(\mathbb{T}^N)$ is a~$2^{d-1}$-\,embedding. This completes the proof.
\end{proof}

\medskip

\subsubsection{Sidon constants II}

The preceding theorem is used to obtain  the following analog of Theorem~\ref{app1}.

\medskip
\begin{Theo} \label{appA}
Let  $1\leq p \leq \frac{2m}{m+1}$ and $r>0$ with $\frac{1}{r}=\frac{1}{p} - \frac{1}{2}$. Then, there is
a~universal constant  $\gamma \geq 1$ such that, for each positive integers $m \leq N$
\[
\frac{1}{\gamma^m}\Big(\frac{N}{m}\Big)^{\frac{m}{r}-\frac{1}{2}} \leq \chi_p\big( \mathcal{B}_{N}^{=m} \big) \leq
\chi_p\big( \mathcal{B}_{N}^{\leq m} \big) \leq \gamma^m \Big( \frac{N}{m}\Big)^{\frac{m}{r}-\frac{1}{2}}\,.
\]
\end{Theo}

\bigskip

Note that in a similar fashion as in Section \ref{sid1}, for each $m\leq N$ one has
\[
\text{$\chi_p\big( \mathcal{B}_{N}^{= m} \big) = \chi_p\big( \mathcal{B}_{N}^{\leq m} \big) =1$}, \quad\, p\in [2, \infty]\,,
\]
and, combining the Bohnenblust--Hille inequality for functions on the Boolen cube from \eqref{BHBoole} with H\"older's inequality,
there exists a~ constant $\gamma \geq 1$ such that
\begin{align*}
\chi_p\big( \mathcal{B}_{N}^{= m} \big)\leq \chi_p\big( \mathcal{B}_{N}^{\leq m} \big) \leq \gamma^m, \quad\,
m\leq N, \,\, \frac{2m}{m+1} \leq p < 2\,.
\end{align*}
We prepare the proof of the preceding theorem with a~lemma similar to Lemma~\ref{unc}.

\medskip

\begin{Lemm} \label{unc1}
Let $1 \leq p \leq \frac{2m}{m+1}$ and $z \in  \mathbb{C}^N$, and
denote by $z^\ast$ the decreasing rearrangement of $z$. Then, for each $m \leq N$ the following estimate holds\,{\rm:}
\begin{align*} \label{crucial2}
(z^\ast_N)^m \leq 4\sqrt{2}e^2 2^{m-1}  \sqrt[r]{m!}\, \sqrt{\log (1+20m)} \, \frac{1}{N^{\frac{m}{r}-\frac{1}{2}}}\,
\|M_z\|\,,
\end{align*}
where  $\frac{1}{r}=\frac{1}{p} - \frac{1}{2}$ and $M_z \colon \mathcal{B}_{N}^{= m} \to \ell_p(\{S\colon |S|= m\})$
is given by
\[
M_z f = \big(\widehat{f}(S)(z^\ast)^S\big)_{|S|=m}, \quad\, f\in \mathcal{B}_{N}^{= m}\,.
\]
\end{Lemm}

\begin{proof}
From Theorem~\ref{MietekM}, we get
\begin{align*}
(z^\ast_N)^{rm}&\sum_{ |S |=m} 1  = \sum_{ |\alpha |=m} \big( (z_N^\ast)\ldots (z_N^\ast) \big)^r \\
& \leq \sum_{ |S |=m} (z^\ast)^{r S} \leq (4e^2 2^{m-1})^r\, \big(1+N\log(1+20 m)\big)^{\frac{r}{2}}\,\|M_z\|^r\,.
\end{align*}
This yields
\begin{align*}
(z^\ast_N)^{rm} \frac{N^m}{m!} \leq (z^\ast_n)^{rm}  \binom{n}{m}  \leq
(4\sqrt{2}e^2 2^{m-1})^r\, \big(N\log(1+20 m)\big)^{\frac{r}{2}}\,\|M_z\|^r\,,
\end{align*}
as required.
\end{proof}

\begin{proof}[Proof of Theorem~\ref{appA}]
\noindent Lower bound: For $\pmb{1} = (1, \ldots ,1) \in \mathbb{C}^{|\{S\colon |S|=m\}|}$ we by definition have
\[
\|M_{\pmb{1}} \colon \mathcal{B}^{=m}_N \to \ell_{p}(\{S; |S|=m\})\| =\chi_p\big(\mathcal{B}^{=m}_N\big)\,.
\]
Then the conclusion follows if we, exactly as in Lemma~\ref{unc}, apply Lemma~\ref{unc1} to $z =\pmb{1}$.
Upper bound: Take $f \in \mathcal{B}^{\leq m}_N$, and interpret it as a~polynomial $F \in \mathcal{P}_{\leq m}(\mathbb{T}^N)$.
Then we know from Theorem~\ref{app1} that
\begin{align*}
\sum_{|S|\leq m}  |\widehat{f}(S)| \leq \gamma^m \Big( \frac{N}{m}\Big)^{\frac{m}{r}-\frac{1}{2}}
\|F\|_{\mathcal{P}_{\leq m}(\mathbb{T}^N)}\leq
\gamma^m (1 + \sqrt{2})^m \Big( \frac{N}{m}\Big)^{\frac{m}{r}-\frac{1}{2}}
\|f\|_{\mathcal{B}^{\leq m}_N}\,,
\end{align*}
where the very last estimate again follows from a result of Klimek in~\cite{K}.
\end{proof}

\medskip

\subsubsection{Monomial convergence II}
Given $V \subset \mathcal{B}$, we define the set of monomial convergence of $V$ by
\[
\mon(V) := \Big\{x \in \R^{\N};\,\ \exists C> 0 \,\,\forall f \in V: \hspace{3mm}
\mbox{$\sum_{S\subset [N]}{|\widehat{f}(S) x^{S}|} \leq C \, \| f\|_{\infty}$} \Big\}\,,
\]
where recall that $x^S := \prod_{n \in S}{x_{n}}$ for each $S \subset [N]$ and for all $x \in \R^{\N}$.
Let us denote by $S \subset_{fin} \mathbb{N}$ the fact that $S$ is a finite subset of $\mathbb{N}$.
We say that a real sequence $(\xi_S)_{S\subset_{fin} \mathbb{N}}$ is multiplicative if
\[
\text{$\xi_R \,\xi_S = \xi_{R \cup S}$ \,for all pairwise disjoint subsets \, $R, S \subset_{fin} \mathbb{N}$}\,.
\]
Note that $(\xi_S)_{S\subset_{fin} \mathbb{N}}$ is multiplicative if and only if there exists $x \in \R^{\N}$,
for all $S \subset_{fin} \mathbb{N}$, we have $x^{S} = \xi_S$. Hence $\mon(V) $  consists exactly of all
multiplicative $\ell_1(\{S\colon S \subset_{fin} \mathbb{N}\})$-multipliers of $V$.

\medskip
Here are some basic properties of $\mon (\mathcal{B})$.
\medskip
\begin{Prop}
\label{Prop:monSetProperties}
Let $x \in \mon (\mathcal{B})$ and $y \in \R^\N$. Then, each of the following conditions yields that
$y\in \mon (\mathcal{B})${\rm:}
\begin{enumerate}
\item[{\rm(i)}] $y$ differs from $x$ in a finite number of entries\,{\rm;}
\item[{\rm(ii)}] $y$ is a~ permutation of $x$\,{\rm;}
\item[{\rm(iii)}] $|y_{n}| \leq |x_{n}|$ for each $n \in \N$.
\end{enumerate}
\end{Prop}

\begin{proof}
To prove the sufficiency of (i), it is enough to assume that $y$ differs from $x$ in one entry, $x_{n} = y_{n}$
for each $n \neq n_{0} \in \N$. Using that $g(x) = x_{n_{0}} f(x)$ also belongs to $\mathcal{B}$ with
$\widehat{g}(S\setminus \{n_{0}\}) = \widehat{f}(S)$ if $n_{0} \in S \subset_{fin} \N$, we get
\[
\sum_{S}{|\widehat{f}(S) y_{S}|} = \sum_{n_{0} \notin S}{|\widehat{f}(S) x^{S}|} + |y_{n_{0}}|\,
\sum_{n_{0} \in S}{|\widehat{f}(S) x_{S\setminus\{ n_{0}\}}|} < \infty\,.
\]
It is a simple observation that for every $f \in \mathcal{B}$ and for each permutation $\sigma$ of the natural
numbers, the function $f_{\sigma}$ defined by
\[
f_{\sigma}((x_n)_{n \in \N}) = f((x_{\sigma(n)})_{n \in \N})
\]
also belongs to $\mathcal{B}$, which proves that condition ${\rm(ii)}$ is sufficient. Finally, if
${\rm(iii)}$ is satisfied, then
\[
\sum_{S}{|\widehat{f}(S) y_{S}|} \leq \sum_{S}{|\widehat{f}(S) x^{S}|} < \infty. \qedhere
\]
\end{proof}

\medskip

Our aim is to find nice descriptions of $\mon(V)$ for $V = \mathcal{B}$, $V = \mathcal{B}^{=m}$, and
$V = \mathcal{B}^{\leq d}$. It is clear that
\[
\mon(\mathcal{B}) \subset \mon(\mathcal{B}^{\leq d}) \subset \mon(\mathcal{B}^{=m}).
\]
Using the Bohnenblust--Hille inequality from \eqref{BHBoole} (together with H\"older's inequality and the
multi-monomial  theorem) gives that, for each $m\in \mathbb{N}$, we have
\[
\ell_{\frac{2m}{m-1}} \subset \mon (\mathcal{B}^{\leq m}) \subset \mon (\mathcal{B}^{= m})\,.
\]
Similar to Theorem~\ref{MON}.${\rm(i)}$ we even have the following full description.

\begin{Theo} For each positive integer $m$ one has
\[
\mon (\mathcal{B}^{=m}) = \mon (\mathcal{B}^{\leq m}) =\ell_{\frac{2m}{m-1}, \infty}\,.
\]
\end{Theo}

\begin{proof}
We first prove that $\text{mon}(\mathcal{B}^{=m}) =\ell_{\frac{2m}{m-1}, \infty}$. Given an $m$-homogeneous function
$f\colon \{-1,1\}^{N} \rightarrow \R$, we can find $F \in H^{m}_{\infty}(\mathbb{T}^\infty)$ with $\widehat{F}(\alpha)
= \widehat{f}(\alpha)$ for all $\alpha \in \{ 0,1\}^{(\N)}$ and $\widehat{F}(\alpha) = 0$ otherwise, and such that
\[
\|F\|_{H_{\infty}(\mathbb{T}^\infty)} \leq 2^{m-1} \, \| f\|_{\mathcal{B}^{=m}}\,,
\]
where for this estimate we  use a result from \cite{V}. This implies that $\mon{H^{m}_{\infty}(\mathbb{T}^\infty)}
\subset \mon (\mathcal{B}^{=m})$, which by Theorem~\ref{MON}.(i) gives the lower inclusion
\[
\ell_{\frac{2m}{m-1}, \infty} \subset \mon (\mathcal{B}^{=m})\,.
\]
Conversely, if $x \in \mon (\mathcal{B}^{=m})$, then by Proposition~\ref{Prop:monSetProperties} also its decreasing
rearrangement  $r = (x_n^\ast) \in \mon (\mathcal{B}^{=m})$.

Thus there is a constant $C > 0$ such that, for each $f \in \mathcal{B}^{=m}$, we get
\[
\sum_{|S| = m}{|\widehat{f}(S)|\,|r_{S}|} \leq C \, \| f\|_{\infty}\,.
\]
We give a probabilistic argument and a Kislyakov-type argument. The probabilistic argument:
Let $A =\{S \subset [N], |S| = m\}$.
By the Kahane-Salem-Zygmund theorem (see \eqref{KSZ} below) there is a choice of signs $(\xi_{S})_{S \in A}$  such that
\[
\sum_{S \in A}{|\xi_{S}|r_{S}} \leq C \, \Big\| \sum_{S \in A}{\xi_{S} x^{S}} \Big\|_{\infty} \ll \, \sqrt{N} \sqrt{\binom{N}{m}}.
\]
Since $(r_{n})_{n \in \N}$ is decreasing, we get
\[
\binom{N}{m} \, r_{N}^{m} \ll \, \sqrt{N} \, \sqrt{\binom{N}{m}}\,,
\]
and so for some constant $K_{m}$ independent of $N$, we have
\[
r_{N}^m \leq C \sqrt{N} \binom{N}{m}^{-\frac{1}{2}} \ll_m \frac{\sqrt{N}}{N^{m/2}} \,.
\]

\bigskip
\noindent {The Kislyakov argument:}
Consider the canonical embeddings
\[
\mathcal{B}^{=m} \longrightarrow \mathcal{P}_{m}([-1,1])^N) \longrightarrow \mathcal{P}_{m}(\mathbb{T}^N)
\longrightarrow \ell^{(1+20m)^N}_\infty\,,
\]
where the first two embedding  are the canonical ones and the last comes from \eqref{bernd}. The first one is isometric,
the second one  $2^{m-1}$-isomorphic, and the third $2$-isomorphic. Then, we deduce from Theorem~\ref{Thm:SK} that
\[
r_N^{2m}\binom{N}{m} \leq \sum_{S \in A} |r^S|^2 \ll_m  N\,.
\]
Since $\Big( \frac{N}{m}\Big)^m \leq \binom{N}{m}$, the argument completes. Finally, we prove that for each
$d \in \N$, we have $\mon (\mathcal{B}^{\leq d}) = \ell_{\frac{2d}{d-1}}$. Using the fact that, for every
$f \in \mathcal{B}^{\leq d}_N$, we get by \cite{K},
\[
\|f\|_\infty \leq (1 +\sqrt{2})^m \|f_m\|_\infty\,,
\]
where $f_m = \sum_{|S|=m} \widehat{f}(S) \chi_S$ denotes the $m$-homogeneous part of $f$, it easily follows that
\[
\ell_{\frac{2d}{d-1}, \infty} \subset \bigcap_{m=1}^{d}{\mon (\mathcal{B}^{=m})} \subset \mon (\mathcal{B}^{\leq d})
\subset \mon (\mathcal{B}^{=d}) \subset \ell_{\frac{2d}{d-1}, \infty}\,.\qedhere
\]
\end{proof}

Let us turn to the description of $\mon (\mathcal{B})$. We easily obtain
\begin{equation} \label{l2}
\ell_{2} \subset \mon (\mathcal{B})\,,
\end{equation}
using the Cauchy--Schwarz inequality
\[
\sum_{S}{|\widehat{f}(S) x^{S}|} \leq \Big( \sum_{S}{|\widehat{f}(S)|^{2}} \Big)^{\frac{1}{2}}
\Big( \sum_{S}{|x^{S}|^2} \Big)^{\frac{1}{2}} \leq \|f\|_{\infty}  \Big(\prod_{n=1}^{\infty}{(1 + |x_{n}|^2)\Big)^{\frac{1}{2}}}\,.
\]
Regarding the estimations from above, using the results for the
$m$-homogeneous functions, we have
\begin{equation}\label{equa:monomialAbove1}
\mon (\mathcal{B}) \subset \bigcap_{m\in \N}{\ell_{\frac{2m}{m-1}, \infty}}.
\end{equation}
This result can be improved.

\begin{Prop} \label{Prop:monomialAbove1}
For every $ x \in \mon (\mathcal{B})$ one has
\begin{equation*}\label{equa:monomialAbove3}
\sup_{N \in \N}{\, \frac{1}{\sqrt{N}}  \sum_{n=1}^{N}{|x_{n}|}} < +\infty\,,
\end{equation*}
and in particular
\begin{equation*}
\mon (\mathcal{B}) \subset \ell_{2, \infty}.
\end{equation*}
\end{Prop}

\begin{proof}
Using the majority function $\Maj_{N}(x)$ (in fact, only its $1$-homogeneous part), we know \cite{BooleanRyan} that for all $S$ with $|S|=1$
\[
\widehat{\Maj}_{N}(S) = \sqrt{\frac{2}{\pi}} \frac{1}{\sqrt{N}}\,.
\]
Hence, for all $x \in \mon (\mathcal{B})$ and every $N \in \mathbb{N}$, we have
\[
\sum_{n=1}^{N}{\sqrt{\frac{2}{\pi}} \frac{1}{\sqrt{N}} \, |x_{n}|}
= \sum_{|S| = 1}{|\widehat{\Maj}_{N}{(S)}| \, |x^{S}|} \leq \sum_{S \subset [N]}{|\widehat{\Maj}_{N}{(S)}| \, |x^{S}|} \leq C_{x}\,,
\]
the first estimate. Since the decreasing rearrangement $x^\ast$ by Proposition \ref{Prop:monSetProperties} also belongs to
$\mon (\mathcal{B})$, the 'in particular' follows.
\end{proof}

\medskip

On the other hand, the use of Kahane--Salem--Zygmund inequality (for Boolean functions, see e.g. \cite[Lemma 3.1]{DefantMastyloPerezA})
does not improve the condition from Proposition~\ref{Prop:monomialAbove1}. Indeed, this inequality   yields that for every $N \in \N$
and each family $(c_{S})_{S \subset [N]}$ in $\mathbb{R}$ there is a~choice of signs $(\xi_{S})_{S \subset [N]}$ such that
\begin{align}  \label{KSZ}
\sum_{S \subset [N]}{|c_{S}| |x^{S}|} \leq C_{x} \, \Big\| \sum_{S \subset [N]}{\xi_{S} c_{S} x^{S}} \Big\|_{\infty}
\leq C_{x} \, 6 \sqrt{\log{2}} \, \sqrt{N}\, \Big( \sum_{S \subset [N]}{|c_{S}|^{2}} \Big)^{1/2}\,.
\end{align}
Taking the supremum over all $(c_{S})_{S \subset [N]}$ with $\ell_{2}$-norm equal to one, we deduce that
\begin{align}  \label{KSZKSZ}
\prod_{n=1}^{N}{(1 + x_{n}^{2})} = \Big(\sum_{S \subset [N]}{|x^{S}|^{2}}\Big)^{1/2} \leq C_{x} \, 6 \sqrt{\log{2}} \, \sqrt{N}.
\end{align}
Since $(x_{n})_{n \in \N}$ converges to zero, we can find a positive constant $\alpha > 0$ such that $\exp{(\alpha |x_{n}|^{2})}
\leq 1 + |x_{n}|^{2}$ for every $n \in \N$, so that
\[
\exp{( \alpha \sum_{n=1}^{N}{|x_{n}|^{2}} )} \leq \prod_{n=1}^{N}{(1 + x_{n}^{2})} \leq C_{x} \, 6 \sqrt{\log{2}} \, \sqrt{N}.
\]
It follows that
\begin{equation}
\sup_{N \in \N}{ \, \frac{1}{\log{N}}\sum_{n=1}^{N}{|x_{n}|^{2}}} < +\infty.
\end{equation}
But this condition is weaker than what we got in Proposition \ref{Prop:monomialAbove1}, since
\[
\frac{1}{\log{N}} \sum_{n=1}^{N}{|x_{n}|^{2}} \leq \frac{1}{\log{N}} \sum_{n=1}^{N}{|x_{n}^\ast|^{2}}
\leq \frac{\| x\|_{\ell_{2, \infty}}^{2}}{\log{N}} \sum_{n=1}^{N}{\frac{1}{n}} \leq \| x\|_{\ell_{2, \infty}}^{2}.
\]

\bigskip

Finally, we establish the following analog of statement $(2)$ from Theorem~\ref{MON}.

\begin{Prop}
For each $x \in \mon (\mathcal{B})$ one has
\[
\limsup_{N \to \infty}  \frac{1}{\log N} \sum_{n=1}^N (x^\ast_n)^2 \leq 1\,.
\]
\end{Prop}

\begin{proof}
We write $r=(r_{n})_{n \in \N}$ for the decreasing rearrangement of $(|x_{n}|)_{n \in \N}$. Then
\[
\sup_{N} \big\|M_r\colon  \mathcal{B}_N \to \ell_1(\{S \colon S \subset [N] \})\big\| < \infty\,,
\]
and hence by Theorem~\ref{MietekM}.$(i)$ we have that there is a constant $C = C(r)$ such that for all $N$
\[
\sum_{ S \subset [N]}{r_{S}^{2}} \leq C \, N;
\]
again, there is an alternative proof using the  Kahane--Salem--Zygmund inequality (in the form of \eqref{KSZ}
and \eqref{KSZKSZ}). We claim that
\[
\frac{(r_{m}^{2} + \ldots + r_{N}^{2})^{m}}{m!} \leq \sum_{|S| = m}{r_{S}^{2}}.
\]
We postpone the proof of the claim to end. The claim yields by Stirling's inequality that there is a~positive
constant $C'$ depending just on $r$ such that
\[
\frac{r_{m}^{2} + \ldots + r_{N}^{2}}{m} \leq (C')^{\frac{1}{m}} m^{\frac{1}{2m}} \, \frac{N^{\frac{1}{m}}}{e}\,.
\]
Putting $m=\log{N}$, we arrive to the inequality
\[
\limsup_{N}{\frac{r_{\log{N}}^{2} + \ldots + r_{N}^{2}}{\log{N}}} \leq 1\,.
\]
Since $r_{n}^{2}$ converges to zero, we immediately conclude that
\[
\limsup_{N}{\frac{r_{1}^{2} + \ldots + r_{N}^{2}}{\log{N}}} \leq 1\,.
\]
To prove the claim, note that every $S \subset \N$ with $|S| = m$ is determined by the subset $S_{1}$ of elements $n$
with $n < m$ and the subset $S_{2}$ of those with $n \geq m$. Then, we can find unique finite sequences
$m \leq i_{1}, i_{2}, \ldots, i_{k} \leq N$ and $m_{1}, m_{2}, \ldots, m_{k}$ in $\N$ with $m_{1}+ \ldots + m_{k}=m$
and such that
\begin{align*}
S^{c} \cap \{n; \,  n < m \} & = \{m_{1},  m_{1} + m_{2}, \ldots,  m_{1} + \ldots + m_{k-1}\}\\
S \cap \{n\; \, n \geq m \} & = \{i_{1} < i_{2} < \ldots < i_{k}\}\,.
\end{align*}
We can then write $r_{S}$ as
\[
r_{S} = r_{1} \cdot \ldots \cdot r_{m_{1}-1} \cdot r_{i_{1}} \cdot r_{m_{1}+1}
\cdot \ldots \cdot r_{m_{1} + m_{2} - 1}\cdot  r_{i_{2}} \cdot r_{m_{1} + m_{2} + 1} \cdot \ldots \]
and using that $(r_{n})_{n \in \N}$ is non-increasing, we deduce that
\[
r_{S} \geq r_{i_{1}}^{m_{1}} \cdot r_{i_{2}}^{m_{2}} \cdot \ldots \cdot r_{i_{k}}^{m_{k}}\,.
\]
Therefore
\begin{align*} m! \sum_{|S| = m}{r_{S}^{2}} & \geq \sum_{\substack{ (\alpha_{m}, \ldots, \alpha_{N}) \in \mathbb{N}_{0}^{N-m+1}
|\alpha| = m}}{\binom{m}{\alpha} r_{m}^{\alpha_{m}} \cdot r_{m+1}^{\alpha_{m+1}}\cdot \ldots \cdot r_{N}^{\alpha_{N}} } \\
\smallskip
& = (r_{m} + \ldots + r_{N})^{m} \,,
\end{align*}
and the proof completes.
\end{proof}

\vspace{2.5 mm}

\noindent
Institut f\"ur Mathematik \\
Carl von Ossietzky Universit\"at \\
Postfach 2503 \\
D-26111 Oldenburg, Germany

\vspace{0.5 mm}

\noindent E-mail: {\tt andreas.defant@uni-oldenburg.de} \\

\vspace{2 mm}

\noindent Faculty of Mathematics and Computer Science \\
Adam Mickiewicz University, Pozna\'n \\
Uniwersytetu Pozna{\'n}skiego 4 \\
61-614 Pozna{\'n}, Poland

\vspace{0.5 mm}

\noindent E-mail: {\tt mastylo@amu.edu.pl} \\

\vspace{2 mm}

\noindent
Departamento de Matem\'{a}tica Aplicada I\\
Escuela T\'{e}cnica Superior de Ingenieros Industriales\\
Universidad Nacional de Educación a Distancia (UNED) \\
28040 Madrid, Spain

\vspace{0.5 mm}

\noindent
E-mail: {\tt antperez@ind.uned.es}
\end{document}